%% file: arXivIFBApplicationsMainFile.tex
\documentclass[a4paper, 12pt]{article}
\usepackage{amsmath,amsthm,amssymb, authblk}
\usepackage[hyperref, usenames, svgnames,x11names, dvipsnames]{xcolor}
\usepackage[ocgcolorlinks, colorlinks=true, citecolor=Cerulean, linkcolor=Bittersweet, urlcolor=blue]{hyperref}

\theoremstyle{definition}
\newtheorem{theorem}{Theorem}[section]
\newtheorem{thm}{Theorem}[section]

\newtheorem{lem}[theorem]{Lemma}

\newtheorem{rem}[theorem]{Remark}

 \newtheorem{defn}[theorem]{Definition}
  \newtheorem{definition}[theorem]{Definition}
 \newtheorem{ass}[theorem]{Assumption} 
  \newtheorem{defnass}[theorem]{Assumption and Definition} 
  \newtheorem{asss}[theorem]{Assumptions} 
  \newtheorem{remark}[theorem]{Remark}
    \newtheorem{corollary}[theorem]{Corollary}
     \newtheorem{lemma}[theorem]{Lemma}

\numberwithin{equation}{section}

\newcommand{\norm}[2]{\left\lVert #1\right\rVert_{#2}}

\newcommand{\md}{\partial^\bullet}
\newcommand{\Vs}{V_0}
\newcommand{\Vt}{V(t)}

\newcommand{\Vms}{V_0^{*}}
\newcommand{\Vmt}{V^{*}(t)}
\newcommand{\Hs}{H_0}
\newcommand{\Ht}{H(t)}
\newcommand{\Hms}{H_0^{*}}
\newcommand{\Hmt}{H^{*}(t)}
\newcommand{\Xs}{X_0}
\newcommand{\Xt}{X(t)}
\newcommand{\Xms}{X_0^{*}}
\newcommand{\Xmt}{X^{*}(t)}
\newcommand{\grad}{\nabla}
\newcommand{\sgrad}{\nabla_\Gamma}

\newcommand{\sgradt}{\nabla_{\Gamma(t)}}
\newcommand{\lap}{\Delta}
\newcommand{\slap}{\Delta_\Gamma}

\newcommand{\symbolForLittlec}{\lambda}
\newcommand{\symbolForBigC}{\Lambda}

\begin{document}




\title{On some linear parabolic PDEs on moving hypersurfaces}

\author{Amal Alphonse, Charles M. Elliott, and Bj\"orn Stinner}
\affil{Mathematics Institute\\ University of Warwick\\ Coventry CV4 7AL\\ United Kingdom}

\maketitle


\begin{abstract}
{We consider existence and uniqueness  for several examples of linear  parabolic equations formulated on moving hypersurfaces. Specifically, we study in turn a surface heat equation, an equation posed on a bulk domain, a novel coupled bulk-surface system and an equation with a dynamic boundary condition. In order to prove the well-posedness, we make use of an abstract framework presented in a recent work by the authors which dealt with the formulation and well-posedness of linear parabolic equations on arbitrary evolving Hilbert spaces. Here, after recalling all of the necessary concepts and theorems, we show that the abstract framework can applied to the case of evolving (or moving) hypersurfaces, and then we demonstrate the utility of the framework to the aforementioned problems.}


\end{abstract}
\section{Introduction}
The analysis and numerical simulation of solutions of partial differential equations on moving hypersurfaces is a prominent area of research \cite{Bonaccorsi, CorRod13, dziuk_elliott, actanumerica, reusken, similar} with many varied applications. Models of certain biological or physical phenomena can be more relevant if formulated on evolving domains (including hypersurfaces); for example, see \cite{Barreira2011,Venkataraman2011,ElliottStinnerVankataraman} for studies of biological pattern formation and cell motility on evolving surfaces, \cite{Garcke2014} for the modelling of surfactants in two-phase flows using a diffuse interface, \cite{Eilks2008} for the modelling and numerical simulation of dealloying by surface dissolution of a binary alloy (involving a forced mean curvature flow coupled to a Cahn--Hilliard equation. In these examples,  the evolving surface is an unknown, giving rise to a free boundary problem. The well-posedness of certain surface parabolic PDEs has been considered in work such as \cite{dziuk_elliott, reusken, vierling}. In \cite{dziuk_elliott}, a Galerkin method was utilised with the pushedforward eigenfunctions of a Laplace--Beltrami operator forming part of the Galerkin ansatz. In \cite{reusken}, the authors make use of the Banach--Ne\v{c}as--Babu\v{s}ka theorem with similar function spaces and results to those that we use, and in \cite{vierling}, a weak form of a surface PDE is pulled back onto a reference domain to which a standard existence theorem is applied. 

The purpose of this paper is twofold: first, to give an account of how an abstract framework that we developed in \cite{AlpEllSti} to handle linear parabolic equations on \emph{abstract} evolving Hilbert spaces can be applied to the case of Lebesgue--Sobolev--Bochner spaces on moving hypersurfaces (and domains), and second, to use the power of this framework to study four different parabolic equations posed on moving hypersurfaces. The first two problems we consider are fairly standard and help to familiarise the concepts, and the last two are novel and are of interest in their own right.

In \cite{AlpEllSti}, under certain assumptions on families of Hilbert spaces parametrised by time, we defined Bochner-type functions spaces (which are generalisations of spaces defined in \cite{vierling}) and an analogue of the usual abstract weak time derivative which we called the weak material derivative, and then we proved well-posedness for a class of parabolic PDEs under some assumptions on the operators involved. A regularity in time result was also given. All of this was done in an abstract Hilbert space setting. We believe that using this approach for problems on moving hypersurfaces is natural and elegant. The concepts and results presented here can also be used as a foundation to study nonlinear equations on evolving surfaces, which can arise from free boundary problems.

\paragraph{Outline}We start in \S \ref{sec:evolvingHypersurfaces} by discussing (evolving) hypersurfaces and some functions spaces, and we formulate the four problems of interest. In \S \ref{sec:abstractFramework}, we recall the essential definitions (of function spaces and of the weak material derivative) and results from \cite{AlpEllSti} without proofs, all in the abstract setting; this section is self-contained in the sense that only the proofs are omitted. In \S \ref{sec:applications}, we discuss in detail realisations of the abstraction to the concrete case of moving domains (which are a special case of evolving flat hypersurfaces) and evolving curved hypersurfaces, i.e., we show that the framework in \S \ref{sec:abstractFramework} is applicable for moving hypersurfaces. Then, we finish in \S \ref{sec:weakFormulationAndWellPosedness} by proving the well-posedness of the four problems introduced in \S \ref{sec:evolvingHypersurfaces}.

\paragraph{Notation and conventions}We fix $T \in (0,\infty)$. When we write expressions such as $\phi_{(\cdot)}u(\cdot)$, our intention usually is that both of the dots $(\cdot)$ denote the same argument; for example, $\phi_{(\cdot)}u(\cdot)$ will come to mean the map $t \mapsto \phi_tu(t).$ The notation $X^*$ will denote the dual space of a Hilbert space $X$ and $X^*$ will be equipped with the usual induced norm $\norm{f}{X^*} = \sup_{x \in X\backslash \{0\}} \langle f, x \rangle_{X^*,X} / \norm{x}{X}$. We may reuse the same constants in calculations multiple times. Integrals will usually be written as $\int_S f(s)$ instead of $\int_S f(s)\;\mathrm{d}s$ unless to avoid ambiguity. Finally, we shall make use of standard notation for Bochner spaces.
\section{Formulation of the equations}\label{sec:evolvingHypersurfaces}
\input{arXivIFBEvolvingHypersurfaces}
\section{Abstract framework}\label{sec:abstractFramework}
\input{arXivIFBAbstractFramework}

\section{Function spaces for evolving hypersurfaces}\label{sec:applications}
\input{arXivIFBApplications}
\subsection*{Acknowledgements}
Two of the authors (A.A. and C.M.E.) were participants of the Isaac Newton Institute programme \emph{Free Boundary Problems and Related Topics} (January -- July 2014) when this article was completed. A.A. was supported by the Engineering and Physical Sciences Research Council (EPSRC) Grant EP/H023364/1 within the MASDOC Centre for Doctoral Training. The authors would like to express their gratitude to the anonymous referees for their careful reading and valuable feedback.

\end{document}

%% file: arXivIFBEvolvingHypersurfaces.tex
As mentioned, we want to showcase four problems that demonstrate the applicability of our theory in different situations, starting with a surface heat equation on an {evolving compact hypersurface without boundary}, and the following on an {evolving domain}: a {bulk equation}, a {coupled bulk-surface system} and a problem with a {dynamic boundary condition}. To formulate these problems, we obviously first need to discuss hypersurfaces and Sobolev spaces defined on hypersurfaces. For reasons of space we shall only briefly touch upon the theory here and refer the reader to \cite{actanumerica, deckelnick, wloka, gilbarg, ss} for more details on analysis on surfaces; we emphasise the text \cite{ss} which contains a detailed overview of the  essential facts.
\subsection{Evolving hypersurfaces and Sobolev spaces}
\paragraph{Hypersurfaces}Recall that $\Gamma$ is an $n$-dimensional $C^k$ \emph{hypersurface} in $\mathbb{R}^{n+1}$ if for each $x \in \Gamma$, there is an open set $U \subset \mathbb{R}^{n+1}$ with $x \in U$ and a function $\Psi \in C^k(U)$ with $\nabla \Psi\neq 0$ on $\Gamma \cap U$ and 
\begin{equation*}
\Gamma \cap U = \{ x \in U \mid \Psi(x) = 0\}.
\end{equation*}
A \emph{parametrised $C^k$ hypersurface} in $\mathbb{R}^{n+1}$ is a map $\psi \in C^k(Y;\mathbb{R}^{n+1})$ where $Y \subset \mathbb{R}^n$ is a connected open set with $\text{rank}(D\psi(y)) = n$ for all $y \in Y$. Locally, parametrised hypersurfaces and hypersurfaces are the same \cite[Chapter 15]{thorpe}. We call $\Gamma$ a $C^k$ \emph{hypersurface with boundary} $\partial\Gamma$ if $\Gamma \backslash \partial\Gamma$ is a $C^k$ hypersurface and if for every $x \in \partial\Gamma$, there exists an open set $U \subset \mathbb{R}^{n+1}$ with $x \in U$ and a homeomorphism $\psi\colon H \to \Gamma \cap U$, where $H := B_1(0) \cap \{y =(y_1, ..., y_n) \in \mathbb{R}^n \mid y_n \leq 0\},$ with $\psi(0) = x$ and
\begin{enumerate}
\item $\text{rank}(D \psi(y)) = n$ for all $y \in H$
\item $\psi(B_1(0) \cap \{y =(y_1, ..., y_n) \in \mathbb{R}^n \mid y_n < 0\}) \subset \Gamma \backslash \partial\Gamma$
\item $\psi(B_1(0) \cap \{y =(y_1, ..., y_n) \in \mathbb{R}^n \mid y_n = 0\}) \subset \partial\Gamma.$
\end{enumerate}
See  \cite[Chapter 20]{thorpe}. A \emph{compact hypersurface} has no boundary. We say $\Gamma$ is a \emph{compact hypersurface with boundary} $\partial\Gamma$ if $\Gamma$ is a hypersurface with boundary $\partial\Gamma$ and $\Gamma \cup \partial\Gamma$ is compact. Throughout this work we assume that $\Gamma$ is orientable with unit normal $\nu$. We say $\Gamma$ is \emph{flat} if the normal $\nu$ is same everywhere on $\Gamma$.
\paragraph{Sobolev spaces}Suppose that $\Gamma$ is an $n$-dimensional compact $C^k$ hypersurface in $\mathbb{R}^{n+1}$ with $k \geq 2$ and smooth boundary $\partial \Gamma$. We can define $L^2(\Gamma)$ in the natural way: it consists of the set of measurable functions $f\colon\Gamma \to \mathbb{R}$ such that
\[\norm{f}{L^2(\Gamma)} := \left(\int_{\Gamma}|f(x)|^2\;\mathrm{d}\sigma(x)\right)^{\frac 1 2} < \infty,\]
where $\mathrm{d}\sigma$ is the surface measure on $\Gamma$ (which we often omit writing).
We will use the notation $\sgrad=(\underline{D}_1 , ..., \underline{D}_{n+1} )$ to stand for the surface gradient on a hypersurface $\Gamma$, and $\slap := \sgrad \cdot \sgrad$ will denote the Laplace--Beltrami operator. The integration by parts formula for functions $f \in C^1(\overline{\Gamma}; \mathbb{R}^{n+1})$ is
\[\int_{\Gamma}\sgrad \cdot f = \int_{\Gamma}f\cdot H\nu + \int_{\partial\Gamma}f\cdot\mu\]
where $H$ is the mean curvature and $\mu$ is the unit conormal vector which is normal to $\partial\Gamma$ and tangential to $\Gamma.$ 
Now if $\psi \in C_c^1(\Gamma)$, then this formula implies
\[\int_{\Gamma}f\underline{D}_i \psi=-\int_{\Gamma}\psi\underline{D}_i f  + \int_{\Gamma}f\psi H\nu_i \qquad \text{for $i=1, ..., n+1$},\]
with the boundary term disappearing due to the compact support. This relation is the basis for defining weak derivatives. We say $f \in L^2(\Gamma)$ has weak derivative $g_i=:\underline{D}_i f \in L^2(\Gamma)$ if for every $\psi \in C_c^1(\Gamma)$, 
\[\int_{\Gamma}f\underline{D}_i \psi=-\int_{\Gamma}\psi g_i  + \int_{\Gamma}f\psi H\nu_i\]
holds. Then we can define the Sobolev space
\[H^{1}(\Gamma) = \{ f \in L^2(\Gamma) \mid \underline{D}_i f \in L^2(\Gamma), i = 1, ..., n+1\}\]
with $\norm{f}{H^1(\Gamma)}^2 := \norm{f}{L^2(\Gamma)}^2 + \norm{\sgrad f}{L^2(\Gamma)}^2.$ 
The above applies to compact hypersurfaces too; in this case the boundary terms in the integration by parts are simply not there. 
We write $H^{-1}(\Gamma)$ for the dual space of $H^1(\Gamma)$ when $\Gamma$ is a compact hypersurface.

We shall also need a fractional-order Sobolev space. Let $\Omega \subset \mathbb{R}^n$ be a bounded Lipschitz domain with boundary $\partial \Omega$. Define the space
\[H^{\frac 1 2}(\partial\Omega) = \{ u \in L^2(\partial\Omega) \mid  \int_{\partial\Omega}\int_{\partial\Omega}\frac{|u(x) - u(y)|^2}{|x-y|^{n}}\;\mathrm{d}\sigma(x)\mathrm{d}\sigma(y) < \infty\}.\]
This is a Hilbert space with the inner product
\begin{align*}
(u,v)_{H^{\frac 12}(\partial\Omega)} &= \int_{\partial\Omega}u(x)v(x)\;\mathrm{d}\sigma(x) + \int_{\partial\Omega}\int_{\partial\Omega}\frac{(u(x) - u(y))(v(x) - v(y))}{|x-y|^{n}}\;\mathrm{d}\sigma(x)\mathrm{d}\sigma(y).
\end{align*}
See \cite[\S 2.4]{ss} and \cite[\S 3.2]{demengel} for details. The notation
\[|u|_{H^{\frac 1 2}(\partial\Omega)} =\left(\int_{\partial\Omega}\int_{\partial\Omega}\frac{|u(x) - u(y)|^2}{|x-y|^{n}}\;\mathrm{d}\sigma(x)\mathrm{d}\sigma(y)\right)^{\frac 1 2}\]
for the seminorm is convenient. Now, recall the standard Green's formula:
\[ \int_{\partial \Omega}\frac{\partial v}{\partial \nu}w= \int_\Omega \grad v \grad w + \int_{\Omega}w\lap v \qquad \forall v \in H^2(\Omega),\;\forall w \in H^1(\Omega).\]
{When $\Omega$ is of class $C^1$, this formula leads us to define a (weak) normal derivative for functions $v \in H^1(\Omega)$ with $\lap v \in L^{2}(\Omega)$ as the element ${\partial v}/{\partial \nu} \in H^{-\frac{1}{2}}(\partial\Omega) := (H^{\frac 12}{(\partial\Omega)})^*$ determined by}
\begin{equation}\label{eq:formulaForNormalDerivative}
\left\langle \frac{\partial v}{\partial \nu}, w \right \rangle_{H^{-\frac 1 2}(\partial\Omega), H^{\frac 1 2}(\partial\Omega)} := \int_\Omega \grad v \grad \mathbb E(w)+ \int_{\Omega}\mathbb{E}(w)\lap v \quad \forall w \in H^{\frac 12}(\partial\Omega),
\end{equation}
where $\mathbb E(w) \in H^1(\Omega)$ is an extension of $w \in H^{\frac 1 2}(\partial\Omega)$; the functional ${\partial v}/{\partial \nu}$ is independent of the extension used for $w.$ See \cite[\S 5.5.1]{demengel} for more details on this.

\paragraph{Evolving hypersurfaces}We say that $\{\Gamma(t)\}_{t \in [0,T]}$ is an \emph{evolving hypersurface} if for every $t_0 \in [0,T]$, there exist open sets $I=(t_0-\delta, t_0+ \delta)$ for some $\delta > 0$ and $U \subset \mathbb{R}^{n+1}$ and a map $\Psi\colon I \times U \to \mathbb{R}$ such that  $\grad \Psi(t,x) \neq 0$ for $x \in \Gamma(t)$ and $t \in I$, and 
\[\Gamma(t) \cap U = \{x \in U \mid \Psi(t,x)=0\} \quad\text{for $t \in I$}.\]
 The \emph{normal velocity} of a hypersurface $\Gamma(t):=\{x \in \mathbb{R}^{n+1} \mid \Psi(x,t)=0\}$ defined by a (global) level set function
is given by
\[\mathbf w_{\nu}=-\frac{\Psi_{t}}{|\nabla \Psi|}\frac{\nabla \Psi}{|\nabla \Psi|}.\]  
\begin{rem}
It is important to note that \emph{the normal velocity is sufficient to define the evolution of a compact hypersurface.} However, a parametrised hypersurface would require the prescription of the full velocity 
of the parametrisation.
\end{rem}
\begin{rem}\label{rem:conormalVelocity}
Consider an evolving hypersurface with boundary. In this case, we need the normal velocity of the surface and the conormal velocity of the boundary in order to describe the evolution. The normal velocity of the surface must agree with the normal velocity of the boundary.
\end{rem}
\begin{rem}\label{rem:evolvingDomainsAreFlatHypersurfaces}An evolving bounded domain $\{\Omega(t)\}$ in $\mathbb{R}^n$ can be viewed as an evolving flat hypersurface with boundary $\{\hat{\Omega}(t)\}$ in $\mathbb{R}^{n+1}$. If we embed each $\Omega(t)$ into the same hyperplane of $\mathbb{R}^{n+1}$ (for example, $\hat{\Omega}(t) = \{ (x_1, ..., x_n, 0) \mid (x_1, ..., x_n) \in \Omega(t)\}$), then the normal velocity $\mathbf w_\nu$ of $\hat{\Omega}(t)$ is zero.
\end{rem}
In order to describe the evolution of a hypersurface, it  is also  useful to assume that there exists a map $F(\cdot,t)\colon \Gamma(0) \to \Gamma(t)$  
which is a diffeomorphism for each $t \in [0,T]$ satisfying $F(\cdot,0) \equiv \text{Id}$ and $\frac{d}{dt}F(\cdot,t) = \mathbf w(F(\cdot,t),t).$ Here we  say that $ \mathbf w$ is the   \emph{material velocity field} and write
\begin{equation}\label{eq:materialVelocityField}
\mathbf w = \mathbf w_\nu + \mathbf w_a
\end{equation} 
where $\mathbf w_\nu$ is the given normal velocity of the  evolving hypersurface and $\mathbf w_a$
is a given tangential velocity field. 

In the next two definitions, we suppose that $u$ is a sufficiently smooth function defined on $\{\Gamma(t)\}_{t \in [0,T]}$ (see \S \ref{sec:evolvingCompactHypersurfaces} later).
\begin{defn}[Normal time derivative]\label{rem:velocities}
Suppose that the hypersurface $\{\Gamma(t)\}$ evolves with a normal velocity $\mathbf w_\nu$. The \emph{normal time derivative} is defined by 
\[\partial^{\circ}u := u_t + \grad u \cdot \mathbf w_\nu.\]
\end{defn}
\begin{defn}[Material derivative]
Suppose that the hypersurface $\{\Gamma(t)\}$ evolves with a normal velocity $\mathbf w_\nu$. Given a tangential velocity 
field $\mathbf w_a$, with $\mathbf w$ as in \eqref{eq:materialVelocityField}, the \emph{material derivative} is defined by
\begin{equation}\label{eq:mdRough}
\md u := u_t + \grad u \cdot \mathbf w.
\end{equation}
We also write $\dot u$ for $\md u$. See \cite{gurtin, fluids}.
\end{defn}
\begin{rem}[Velocity fields]\label{rem:differentVelocityFields}It is useful to note that there are different notions of velocities for an evolving hypersurface.
\begin{itemize}
\item Suppose that the velocity $\mathbf w$ of an evolving compact hypersurface is purely tangential (so $\mathbf w\cdot \nu = 0$). 
In this case, material points on the initial surface get transported across the surface over time but \emph{the surface remains the same}. 
One can see this for a sufficiently smooth initial surface $\Gamma_0$ by supposing that $\Gamma_0$ is the zero-level set of a function $\Psi\colon \mathbb{R}^{n+1} \to \mathbb{R}$: 
\[\Gamma_0 = \{ x \in \mathbb{R}^{n+1} \mid \Psi(x) = 0\}.\]
Let $P$ be a material point on $\Gamma_0$ and $\gamma(t)$ denote the position of $P$ at time $t$, with $\gamma(t) \in \Gamma(t)$. 
Then a purely tangential velocity means that $\grad \Psi(\gamma(t))\cdot \gamma'(t)=0,$
but this is precisely 
\[\frac{d}{dt}\Psi(\gamma(t))=0,\]
so the point persists in being a zero of the level set. Since $P$ was arbitrary, we conclude that $\Gamma(t)$ coincides 
with $\Gamma_0$ for all $t \in [0,T],$ i.e., $\Gamma(t) = \{ x \in \mathbb{R}^{n+1} \mid \Psi(x) = 0\}.$

\item
In applications, there may be a  \emph{physical velocity}
\[\mathbf w_\nu + \mathbf w_\tau,\]
where $\mathbf w_\nu$ is the normal component and $\mathbf w_\tau$ is the tangential component. The tangential velocity may 
be associated with the motion of physical material points and may be relevant to the mathematical models of processes on the 
surface.

\item
The velocity field \eqref{eq:materialVelocityField} defines the path that points on the initial surface take with respect to the mapping $F$.
In finite element analysis, it may be necessary to choose the tangential velocity $\mathbf w_a$ in an 
ALE approach so as to yield a shape-regular or adequately refined mesh. See \cite{elliottstyles} 
and  \cite[\S 5.7]{actanumerica} for more details on this. One may wish to use this physical tangential velocity to define the map $F$. In writing down PDEs on evolving surfaces it is important to distinguish these notions. 
 
 \item In certain situations, it can be useful to consider on an evolving surface 
 a \emph{boundary velocity} $\mathbf w_b$ which we can extend (arbitrarily) to the interior. In the case of flat 
 hypersurfaces with $\mathbf w_\nu \equiv 0$ (this is the case when an evolving domain in $\mathbb{R}^n$ is 
 viewed like in Remark \ref{rem:evolvingDomainsAreFlatHypersurfaces}), the conormal component of the arbitrary velocity 
 must agree with the conormal component of the boundary velocity $\mathbf w_b$, otherwise the velocities map to two different surfaces.
\end{itemize}
\end{rem}
\subsection{The equations}\label{sec:pdes}
We now state the equations we will study. Three of the problems are posed on evolving bounded 
open sets in $\mathbb R^{n}$. 
In this case, we shall denote by $\Omega(t)$ the evolving domain and $\Gamma(t)$ will denote the evolving compact hypersurface $\partial\Omega(t)$. 
In the equations given below, $\mathbf w$ is a velocity field which has a normal component $\mathbf w_\nu$ agreeing with the normal 
velocity of the evolving hypersurface or domain associated to the problem and an arbitrary tangential component $\mathbf w_a$.
\paragraph{Surface heat equation}Suppose we have an evolving compact hypersurface $\Gamma(t)$ that evolves with 
normal velocity $\mathbf{w}_{\nu}$. Given a surface flux $\mathbf q$, we consider the conservation law
\[\frac{d}{dt}\int_{M(t)}u = -\int_{\partial M(t)}\mathbf q\cdot \mathbf \mu\]
on an arbitrary portion $M(t) \subset \Gamma(t)$, where $\mathbf \mu$ denotes the conormal on $\partial M(t)$. Without loss of generality we can 
assume that $\mathbf q$ is tangential. This conservation law implies the pointwise 
equation $u_t + \grad u \cdot \mathbf w_{\nu} + u\sgrad \cdot \mathbf w_{\nu} + \sgrad \cdot \mathbf q = 0.$ Assuming that the 
flux is a combination of a diffusive flux and an advective flux, so that $\mathbf q = - \sgrad u + u\mathbf b_{\tau}$ 
where $\mathbf b_{\tau}$ is an advective  tangential velocity field, we obtain 
$u_t + \grad u \cdot \mathbf w_{\nu} + u\sgrad \cdot \mathbf w_{\nu} -\slap u + \sgrad u \cdot \mathbf b_{\tau} + u\sgrad \cdot \mathbf b_{\tau} = 0.$ 
Setting $\mathbf b = \mathbf w_{\nu} + \mathbf b_{\tau},$ 
and recalling \eqref{eq:mdRough}, 
we end up with the surface heat equation
\begin{equation}\label{eq:surfaceHeatEquationStart}
\begin{aligned}
\dot u - \slap u + u\sgrad \cdot \mathbf  b + \sgrad  u \cdot (\mathbf b-\mathbf w) &= 0\\
u(0) &= u_0
\end{aligned}
\end{equation}
supplemented with an initial condition $u_0 \in L^2(\Gamma_0)$.
\paragraph{A bulk equation}With $f(t)\colon \Omega(t) \to \mathbb{R}$ and 
$u_0\colon \Omega_0 \to \mathbb{R}$ given, consider the boundary value problem
\begin{equation}\label{eq:bulkEquationStart}
\begin{aligned}
\dot u(t)+ (\mathbf b(t)-\mathbf{w}(t))\cdot \grad u(t) + u(t)\grad\cdot \mathbf b(t) - D\lap u(t) &= f(t)\qquad &&\text{on $\Omega(t)$}\\
u(t,\cdot) &= 0 \quad &&\text{on $\Gamma(t)$}\\
u(0,\cdot) &= u_0(\cdot)&&\text{on $\Omega_0$}
\end{aligned}
\end{equation}
where $D > 0$ is a constant and {the physical material velocity $\mathbf b(t)\colon\Omega(t)\to \mathbb R^{n}$ is 
sufficiently smooth with $\norm{\mathbf b(t)}{L^\infty(\Omega(t))} \leq C_1$ and 
$\norm{\grad\cdot \mathbf b(t)}{L^\infty(\Omega(t))} \leq C_2$ for constants $C_1$ and $C_2$ uniform for all almost time}. 
We refer the reader to \cite{CorRod13} for a formulation of balance equations on moving {time-dependent} bulk domains.
\paragraph{A coupled bulk-surface system}In \cite{ranner}, the authors consider the well-posedness of an elliptic coupled
 bulk-surface system on a (static) domain; we now extend this to the parabolic case on an evolving domain. 
 Given $f(t)\colon \Omega(t) \to \mathbb{R}$, $g(t)\colon \Gamma(t) \to \mathbb{R},$ $u_0 \in H^1(\Omega_0)$ 
 and $v_0 \in H^1(\Gamma_0)$, we want to find solutions $u(t)\colon\Omega(t) \to \mathbb{R}$ and $v(t)\colon\Gamma(t) \to \mathbb{R}$ 
 of the coupled bulk-surface system
\begin{alignat}{3}
\dot u - \lap_\Omega u + u \grad_\Omega \cdot \mathbf w &= f &&\text{on $\Omega(t)$}\label{eq:bs_1} \\
\dot v - \slap v + v\sgrad \cdot \mathbf w + \grad_\Omega u \cdot \nu  &= g&&\text{on $\Gamma(t)$}\label{eq:bs_2}\\
\grad_\Omega u \cdot \nu &= \beta v - \alpha u\qquad&&\text{on $\Gamma(t)$}\label{eq:bs_bc}\\
u(0) &= u_0&&\text{on $\Omega_0$}\label{eq:bs_ICforu}\\
v(0) &= v_0&&\text{on $\Gamma_0$}\label{eq:bs_ICforv}
\end{alignat}
where $\alpha$, $\beta > 0$ are constants. 
Note that \eqref{eq:bs_bc} is a Robin boundary condition for $u$ and that we reused the notation $u$ for denoting the trace of $u$. 
We use the physical material velocity to define the mapping $F$ and  assume there is just the one velocity field $\mathbf{w}$ 
which advects $u$ within $\Omega$ and $v$ on $\Gamma$.
\paragraph{A dynamic boundary problem for an elliptic equation}
Given $f(t) \in H^{-\frac 12}(\Gamma(t))$ and $v_0 \in L^2(\Gamma_0)$, we consider the problem of finding a function $v(t)\colon \Omega(t) \to \mathbb{R}$ such that, with $u(t) := v(t)|_{\Gamma(t)}$ denoting the trace,
\begin{equation}\label{eq:PDEstart}
\begin{aligned}
\lap v(t) &= 0 &&\text{on $\Omega(t)$}\\
\dot u(t) + \frac{\partial v(t)}{\partial \nu(t)}+ u(t) &= f(t)  &&\text{on $\Gamma(t)$}\\
u(0) &= v_0 &&\text{on $\Gamma_0$}
\end{aligned}
\end{equation}
holds in a weak sense. Here we assume that $\Gamma(t)$ evolves with the velocity $\mathbf w$ which we suppose is a normal velocity. This is a natural (linearised) extension to evolving domains of the problem considered by Lions in \cite[\S 1.11.1]{lionsquel}.  

In order to formulate these equations in an appropriate weak sense and carry out the analysis, we will need Bochner-type function spaces for evolving hypersurfaces and the associated theory. This is done in the abstract sense in the next section.

%% file: arXivIFBAbstractFramework.tex
The aim of this section is to give meaning to the setting and analysis of parabolic problems of the form $\dot u(t) + A(t)u(t) = f(t)$, where the equality is in $V^*(t)$, with $V(t)$ a Hilbert space for each $t \in [0,T]$. We employ the notations and results of \cite{AlpEllSti} here and give a self-contained account (see \cite{AlpEllSti} for more details). 
\subsection{Evolving spaces}\label{sec:assumptionsEvolution}
We informally identify a family of Hilbert spaces $\{X(t)\}_{t \in [0,T]}$ with the symbol $X$, and given a family of maps $\phi_{t}\colon \Xs \to \Xt$ 
we define the following notion of \textbf{compatibility}.
\begin{definition}[Compatibility]\label{comp}
We say that a pair $(X, (\phi_{t})_{t \in [0,T]})$ is \emph{compatible} if all of the following conditions hold.

For each $t \in [0,T]$, $X(t)$ is a real separable Hilbert space (with $X_0 := X(0)$) and the map $\phi_{t}\colon X_0 \to X(t)$ is a linear homeomorphism such that $\phi_0$ is the identity. We denote by $\phi_{-t}\colon \Xt \to \Xs$ the inverse of $\phi_t.$   Furthermore, we assume there exists a constant $C_X$ independent of $t$ such that
\begin{equation*}
\begin{aligned}
\norm{\phi_t u}{X(t)} &\leq C_X\norm{u}{X_0}&&\forall u \in X_0\\
\norm{\phi_{-t} u}{\Xs} &\leq C_X\norm{u}{\Xt}&&\forall u \in \Xt.
\end{aligned}
\end{equation*}
Finally, we assume continuity of the map $t \mapsto \norm{\phi_t u}{X(t)}$ for all $u \in X_0.$
\end{definition}
We often write the pair as $(X, \phi_{(\cdot)})$ for convenience. We call $\phi_{t}$ and $\phi_{-t}$ the \emph{pushforward} and \emph{pullback} maps respectively. In the following we will assume compatibility of $(X, \phi_{(\cdot)})$. As a consequence, the dual operator of $\phi_t$, denoted $\phi_t^*\colon X^*(t) \to X_0^*$, is itself a linear homeomorphism, as is its inverse $\phi_{-t}^*\colon \Xms \to \Xmt$, and they satisfy
\begin{equation*}
\begin{aligned}
\norm{\phi_t^* f}{X_0^*} &\leq C_X\norm{f}{X^*(t)} &&\forall f \in X^*(t)\\
\norm{\phi_{-t}^* f}{\Xmt} &\leq C_X\norm{f}{\Xms} &&\forall f \in \Xms.
\end{aligned}
\end{equation*}
By separability of $X_0$, we have measurability of the map $t \mapsto \norm{\phi_{-t}^* f}{X^*(t)}$ for all $f \in X^*_0.$
\begin{rem}The maps $\phi_t$ are similar to the Arbitrary Lagrangian Eulerian (ALE) maps ubiquitous in applications on moving domains. See \cite{ALE} for an account of the ALE framework and a comparable set-up. Also, if we define $U(t,s)\colon X(s) \to X(t)$ by $U(t,s) := \phi_t \phi_{-s}$ for $s$, $t \in [0,T]$, it can be readily seen from 
$U(t,r)U(r,s) = U(t,s)$ that the family $U(t,s)$ is a two-parameter semigroup.
\end{rem}
We now define suitable Bochner-type function spaces which are generalisations of those in \cite{vierling}.
\begin{definition}[The spaces $L^2_X$ and $L^2_{X^*}$]
Define the separable Hilbert spaces 
\begin{align*}
L^2_X &= \{u:[0,T] \to \!\!\!\!\bigcup_{t \in [0,T]}\!\!\!\! X(t) \times \{t\}, t \mapsto (\bar u(t), t) \mid \phi_{-(\cdot)} \bar u(\cdot) \in L^2(0,T;X_0)\}\\
L^2_{X^*} &= \{f:[0,T] \to \!\!\!\!\bigcup_{t \in [0,T]}\!\!\!\! X^*(t) \times \{t\},t \mapsto (\bar f(t), t) \mid \phi_{(\cdot)}^* \bar f(\cdot) \in L^2(0,T;X^*_0) \}
\end{align*}
with the inner products
\begin{equation}\label{eq:innerProductOnL2X}
\begin{aligned}
(u, v)_{L^2_X} &= \int_0^T (u(t), v(t))_{X(t)}\;\mathrm{d}t\\
(f, g)_{L^2_{X^*}} &= \int_0^T (f(t), g(t))_{X^*(t)}\;\mathrm{d}t.
\end{aligned}
\end{equation}
\end{definition}
Note that we made an abuse of notation in \eqref{eq:innerProductOnL2X} and identified $u(t) = (\bar u(t), t)$ with $\bar u(t)$ for $u \in L^2_X$, (and likewise for $f \in L^2_{X^*}$); we shall persist with this abuse below. These spaces, to be precise, consist of equivalence classes of functions agreeing almost everywhere in $[0,T]$. The maps
$u \mapsto \phi_{(\cdot)}u(\cdot)$ from $L^2(0,T;X_0)$ to $L^2_X$ and
$f \mapsto \phi_{-(\cdot)}^*f(\cdot)$ from $L^2(0,T;X_0^*)$ to $L^2_{X^*}$
are both isomorphisms between the respective spaces with the equivalence of norms
\begin{equation*}
\begin{aligned}[2]
\frac{1}{C_X}\norm{u}{L^2_X} &\leq \norm{\phi_{-(\cdot)}u(\cdot)}{L^2(0,T;X_0)} \leq C_X\norm{u}{L^2_X}\qquad &&\forall u \in L^2_X\\
\frac{1}{C_X}\norm{f}{L^2_{X^*}} &\leq \norm{\phi_{(\cdot)}^*f(\cdot)}{L^2(0,T;X^*_0)} \leq C_X\norm{f}{L^2_{X^*}}&&\forall f \in L^2_{X^*}.
\end{aligned}
\end{equation*}
\begin{lemma}[Identification of $(L^2_X)^*$ and $L^2_{X^*}$]The dual space of $L^2_{X}$ can be identified with $L^2_{X^*}$, and the duality pairing of $f \in L^2_{X^*}$ with $u \in L^2_X$ is given by
\[\langle f, u \rangle_{L^2_{X^*}, L^2_X} = \int_0^T \langle f(t), u(t) \rangle_{\Xmt, \Xt}\;\mathrm{d}t.\]
\end{lemma}
\begin{definition}[Spaces of pushed-forward  continuously differentiable functions]
Define
\begin{align*}
C^k_X &= \{\xi \in L^2_X \mid \phi_{-(\cdot)}\xi(\cdot) \in C^k([0,T];X_0)\}\quad\text{for $k \in \{0,1,...\}$}\\
\mathcal{D}_X(0,T) &= \{\eta \in L^2_X \mid \phi_{-(\cdot)}\eta(\cdot) \in \mathcal{D}((0,T);X_0)\}\\
\mathcal{D}_X[0,T] &= \{\eta \in L^2_X \mid \phi_{-(\cdot)}\eta(\cdot) \in \mathcal{D}([0,T];X_0)\}.
\end{align*}
\end{definition}
Since $\mathcal{D}((0,T);X_0) \subset \mathcal{D}([0,T];X_0)$, we have $\mathcal{D}_X(0,T) \subset \mathcal{D}_X[0,T] \subset C^k_X.$
\subsection{Evolving Hilbert space structure}
For each $t \in [0,T]$, let $\Vt$ and $\Ht$ be (real) separable Hilbert spaces with $V_0 := V(0)$ and $H_0 := H(0)$ such that $V_0 \subset H_0$ is a continuous and dense embedding. Identifying $H_0$ with its dual $H_0^*$ via the Riesz representation theorem, it follows that 
$V_0 \subset H_0 \subset V_0^*$ is a Gelfand triple.
\begin{asss}\label{asss:compatibilityOfEvolvingHilbertTriple}
The pairs $(H,\phi_{(\cdot)})$ and $(V, \phi_{(\cdot)}|_{V_0})$ are assumed to be compatible for a (given) family of linear homeomorphisms $\{\phi_{t}\}_{t \in [0,T]}$.
%
We simply write $\phi_t$ instead of $\phi_t|_{V_0}$, and we denote the dual operator of $\phi_t\colon V_0 \to \Vt$ by $\phi_t^*\colon \Vmt \to \Vms$; we are not interested in the dual of $\phi_t\colon \Hs \to \Ht.$ 
\end{asss}
See \cite[\S 2.3]{AlpEllSti} for a convenient summary of the meaning of these assumptions. It follows that for each $t \in [0,T]$, $\Vt \subset \Ht$ is continuously and densely embedded.
The results in \S \ref{sec:assumptionsEvolution} tell us that the spaces $L^2_H$, $L^2_V,$ and $L^2_{V^*}$ are Hilbert spaces with the inner product given by the formula \eqref{eq:innerProductOnL2X}. 
It follows upon identification of $L^2_H$ with its dual in the natural manner that 
$L^2_V \subset L^2_H \subset L^2_{V^*}$ is a Gelfand triple. We make use of the formula $\langle f, u \rangle_{L^2_{V^*}, L^2_V} = (f, u)_{L^2_H}$ whenever $f \in L^2_H$ and $u \in L^2_V$.
\subsection{Abstract strong and weak material derivatives}\label{sec:abstractMaterialDerivative}
\begin{definition}[Strong material derivative]For $\xi \in C^1_X$
define the \emph{strong material derivative} $\dot \xi \in C^0_X$ by
\begin{equation}\label{eq:defnStrongMaterialDerivative}
\dot{\xi}(t) := \phi_{t}\left(\frac{d}{dt}(\phi_{-t} \xi(t))\right).
\end{equation}
\end{definition} 
In the evolving surface case, we show in \S \ref{sec:evolvingCompactHypersurfaces} that this abstract formula agrees with \eqref{eq:mdRough}. 
\begin{definition}[Relationship between $H_0$ and $\Ht$]\label{defn:bilinearFormb}
For all $t \in [0,T]$, define the bounded bilinear form $\hat{b}(t;\cdot,\cdot)\colon \Hs \times \Hs \to \mathbb{R}$ by
$\hat{b}(t;u_0,v_0) = (\phi_t u_0, \phi_t v_0)_{\Ht}$ for $u_0$, $v_0 \in \Hs$.
\end{definition}
It follows that for each $t \in [0,T],$ $\hat{b}(t;\cdot,\cdot)$ is an alternative inner product on $\Hs$; thanks to the Riesz representation theorem, there exists a bounded linear operator $T_t\colon H_0 \to H_0$ such that $\hat b(t;u_0,v_0) = (T_tu_0, v_0)_{H_0} = (u_0, T_tv_0)_{H_0}.$ In fact, $T_t \equiv \phi_t^A\phi_t$, where $\phi_t^A\colon H(t) \to H_0$ is the Hilbert-adjoint of $\phi_t\colon H_0 \to \Ht$.
\begin{asss}\label{asss:differentiabilityOfNorm}
For all $u_0$, $v_0 \in H_0$, assume the following: $\theta(t,u_0) := \frac{d}{dt}\norm{\phi_t u_0}{\Ht}^2$ exists classically,  $u_0 \mapsto \theta(t,u_0)$ is continuous, and $|\theta(t,u_0+v_0) - \theta(t,u_0-v_0)|\leq C\norm{u_0}{H_0}\norm{v_0}{H_0}$
where the constant $C$ is independent of $t \in [0,T]$.
\end{asss}
It follows that $\hat{\symbolForLittlec}(t;\cdot,\cdot)\colon \Hs\times \Hs \to \mathbb{R}$ is well-defined by
$\hat{\symbolForLittlec}(t;u_0,v_0) := \frac{d}{dt}\hat{b}(t;u_0,v_0) = \frac{1}{4}\left(\theta(t,u_0+v_0) - \theta(t,u_0-v_0)\right).$
Denote by $\hat \symbolForBigC(t)\colon \Hs \to \Hms$ the map
$\langle \hat \symbolForBigC(t) u_0, v_0 \rangle := \hat \symbolForLittlec(t;u_0, v_0).$
\begin{definition}\label{defn:bilinearFormc}
For $u$, $v \in \Ht$, define the bilinear form $\symbolForLittlec(t;\cdot,\cdot)\colon \Ht \times \Ht \to \mathbb{R}$ by $\symbolForLittlec(t;u,v) := \hat{\symbolForLittlec}(t; \phi_{-t}u, \phi_{-t}v).$
\end{definition}
The map $t \mapsto \symbolForLittlec(t;u(t),v(t))$ is measurable for all $u$, $v \in L^2_H,$ and $\symbolForLittlec(t;\cdot,\cdot)\colon\Ht \times \Ht \to \mathbb{R}$ is bounded independently of $t$: $|\symbolForLittlec(t;u,v)| \leq C\norm{u}{\Ht}\norm{v}{\Ht}.$
\begin{definition}[Weak material derivative]
For $u \in L^2_V$, if there exists a function $g \in L^2_{V^*}$ such that
\[\int_0^T \langle g(t), \eta(t)\rangle_{\Vmt, \Vt} =-\int_0^T (u(t), \dot{\eta}(t))_{\Ht} - \int_0^T \symbolForLittlec(t;u(t), \eta(t))\]
holds for all $\eta \in \mathcal{D}_V(0,T)$, then $g$ is said to be the \emph{weak material derivative} of $u$, and we write $\dot u=g$  or $\md u=g.$ 
\end{definition}
This concept of a weak material derivative is well-defined: if it exists, it is unique, and every strong material derivative is also a weak material derivative. 
\begin{definition}\label{defn:solutionSpace}
We denote by 
$W(V,V^*) = \{ u \in L^2_V \mid \dot u \in L^2_{V^*}\}$
the Hilbert space endowed with the inner product
\[(u, v)_{W(V,V^*)} = \int_0^T (u(t), v(t))_{\Vt} + \int_0^T(\dot u(t), \dot v(t) )_{\Vmt}.\]
\end{definition}

The space $W(V,V^*)$ is deeply linked to the following standard Sobolev--Bochner space.
\begin{definition}[{\cite[\S 25]{wloka}}]
We denote by $\mathcal W(V_0, V_0^*) = \{v \in L^2(0,T;\Vs) \mid v' \in L^2(0,T; \Vms)\}$ the Hilbert space endowed with the inner product 
\[ (u,v)_{\mathcal W(V_0,  V_0^*)} = \int_0^T (u(t), v(t))_{V_0} + \int_0^T(u'(t), v'(t))_{V_0^*}.\]
\end{definition}
In practice, the next assumption is the most difficult to check.
\begin{defnass}\label{ass:spaceW1}
We assume that there is an \emph{evolving space equivalence} between $W(V,V^*)$ and $\mathcal W(V_0, V_0^*)$. This means that $v \in W(V,V^*)$ if and only if $\phi_{-(\cdot)} v(\cdot) \in \mathcal W(V_0, V_0^*),$
and there holds the equivalence of norms
\[C_1\norm{\phi_{-(\cdot)}v(\cdot)}{\mathcal W(V_0, V_0^*)} \leq \norm{v}{W(V,V^*)} \leq C_2\norm{\phi_{-(\cdot)}v(\cdot)}{\mathcal W(V_0, V_0^*)}.\]
\end{defnass}
This assumption holds under the following conditions.
\begin{theorem}\label{thm:spaceW1}Suppose that
\begin{equation}\label{eq:assIso}
u \in \mathcal W(V_0,V_0^*) \quad\text{if and only if}\quad T_{(\cdot)}u(\cdot) \in \mathcal W(V_0,V_0^*)\tag{T1}
\end{equation}
and that there exist operators $\hat{S}(t)\colon \Vms \to \Vms$ and $\hat D(t)\colon \Vs \to \Vms$
 such that for $u \in \mathcal{W}(V_0,V_0^*),$ 
\begin{equation}\label{eq:assDiffT}
(T_t u(t))' = \hat{S}(t)u'(t) + \hat \symbolForBigC(t) u(t)+\hat D(t)u(t) \tag{T2}
\end{equation}
and $\hat S(\cdot)u'(\cdot)$, $\hat D(\cdot)u(\cdot) \in L^2(0,T;V_0^*)$. 
Suppose also that $\hat{S}(t)$, $\hat{S}(t)^{-1},$ and $\hat{D}(t)$
are bounded independently of $t$. Then  $W(V,V^*)$ is equivalent to $\mathcal W(V_0, V_0^*)$ in the sense of Definition \ref{ass:spaceW1}.
\end{theorem}
\begin{remark}If we knew that $T_tv_0 \in V_0$ for every $v_0 \in V_0$, then the assumption \eqref{eq:assDiffT} would follow from \eqref{eq:assIso} with $\langle \hat S(t)f, v \rangle_{\Vms, \Vs} := \langle f, T_tv \rangle_{\Vms, \Vs}$ and $\hat D(t) \equiv 0$.
\end{remark}
\begin{corollary}The space $W(V,V^*)$ is a Hilbert space. We have the embedding $W(V,V^*) \subset C^0_H$ and the inequality 
\[\max_{t \in [0,T]}\norm{u(t)}{H(t)} \leq C\norm{u}{W(V,V^*)}\qquad\text{$\forall u \in W(V,V^*)$}.\]
\end{corollary}
This allows us to define the subspace 
$W_0(V,V^*) = \{ u \in W(V,V^*) \mid u(0) = 0 \}.$
Let $AC([0,T])$ be the space of absolutely continuous functions from $[0,T]$ into $\mathbb{R}$. The following space is needed for the formulation of an assumption for the regularity of the solution.
\begin{definition}\label{defn:galerkinSpace}
We define the space
\[\tilde{C}^1_V = \{u \mid u(t) = \sum_{j=1}^m \alpha_j(t)\chi_j^t, \text{ $m \in \mathbb{N}$, $\alpha_j \in AC([0,T])$ and $\alpha_j' \in L^2(0,T)$}\}.\]
Note that $\tilde C^1_V \subset C^0_V$ and $\tilde C^1_V \subset W(V,V)$.
\end{definition}
If $u \in \tilde C^1_V$ with $u(t) = \sum_{j=1}^m \alpha_j(t)\chi_j^t$ as in the definition then $\dot u(t) = \sum_{j=1}^m \alpha'_j(t)\chi_j^t.$ We cannot use \eqref{eq:defnStrongMaterialDerivative} for the strong material derivative $\dot u$ because $\phi_{-(\cdot)}u(\cdot) \notin  
C^1([0,T];V_0)$ in general.
\begin{definition}Define the space
$W(V,H) = \{ u \in L^2_V \mid \dot u \in L^2_{H}\}.$
\end{definition}
In order to obtain a regularity result, we need to make the following natural assumption, which will also tell us that $W(V,H)$ is a Hilbert space.
\begin{ass}\label{ass:evolvingSpaceEquivalenceStronger}
It is assumed that there exists an evolving space equivalence between $W(V,H)$ and $\mathcal W(V_0,H_0)$.
\end{ass}
This assumption follows if, for example, \eqref{eq:assIso} is altered in the obvious way and the maps $\hat S(t)$ and $\hat D(t)$ of Theorem \ref{thm:spaceW1} satisfy $\hat S(t)\colon H_0 \to H_0$ and $\hat D(t)\colon V_0 \to H_0$, with both maps and $\hat S(t)^{-1}$ being bounded independently of $t$, and if $\hat S(\cdot) u'(\cdot)$, $\hat D(\cdot)u(\cdot) \in L^2(0,T;H_0)$ for $u \in \mathcal{W}(V_0,H_0)$.

\begin{theorem}[Transport theorem and integration by parts]\label{thm:transportTheorem}
For all $u$, $v \in W(V,V^*)$, the map $t \mapsto (u(t),v(t))_{\Ht}$ is absolutely continuous on $[0,T]$ and 
\[\frac{d}{dt}(u(t),v(t))_{\Ht} = \langle \dot u(t), v(t) \rangle_{\Vmt, \Vt}+\langle \dot v(t), u(t) \rangle_{\Vmt, \Vt} + \symbolForLittlec(t;u(t),v(t))\]
for almost every $t \in [0,T]$, hence
there holds the integration by parts formula
\begin{align*}
(u(T)&,v(T))_{H(T)} - (u(0),v(0))_{H_0} \\
&= \int_0^T\langle \dot u(t), v(t) \rangle_{\Vmt, \Vt}+ \langle \dot v(t), u(t) \rangle_{\Vmt, \Vt}
+ \symbolForLittlec(t;u(t),v(t))\;\mathrm{d}t.
\end{align*}
\end{theorem}

\subsection{Well-posedness and regularity}
Continuing with the framework and notation presented in the previous subsections, and reiterating in particular Assumptions \ref{asss:compatibilityOfEvolvingHilbertTriple}, \ref{asss:differentiabilityOfNorm}, and \ref{ass:spaceW1}, 
we showed in \cite{AlpEllSti} the existence, uniqueness, and continuous dependence of solutions $u \in W(V,V^*)$ to equations of the form
\begin{equation}\label{eq:operatorEquation}
\begin{aligned}
L \dot u + A u+ \symbolForBigC u &= f&&\text{in $L^2_{V^*}$}\\
u(0) &= u_0&&\text{in $H_0$},
\end{aligned}\tag{$\textbf{P}$}
\end{equation}
where we identify $(L\dot u)(t) = L(t)\dot u(t)$, $(Au)(t) = A(t)u(t)$ and $(\symbolForBigC u)(t) = \symbolForBigC(t)u(t),$
with $L(t)$ and $A(t)$ being linear operators that satisfy Assumptions \ref{asss:onL} and \ref{asss:aWithoutL} given below, and 
$\symbolForBigC(t)\colon\Ht \to \Hmt$ is defined by $\langle \symbolForBigC(t) v,w  \rangle_{\Hmt, \Ht} := \symbolForLittlec(t;v,w)$
(see Definition \ref{defn:bilinearFormc}). 
\begin{asss}[Assumptions on $L(t)$]\label{asss:onL}
In the following, all constants $C_i$ are positive and independent of $t \in [0,T]$. Assume for all $g \in L^2_{V^*}$ that
\begin{align}
Lg \in L^2_{V^*}\qquad\text{and}\qquad C_1 \norm{g}{L^2_{V^*}} \leq \norm{Lg}{L^2_{V^*}} \leq C_2 \norm{g}{L^2_{V^*}}\label{eq:assLvstarInL2Vstar}\tag{L1}.
\end{align}
Suppose that the restriction $L|_{L^2_H}$  satisfies $L|_{L^2_H}\colon L^2_H \to L^2_H$. We identify $(L|_{L^2_H}h)(t)$ with $L_H(t)h(t),$ and suppose that 
$L_H(t)\colon \Ht \to \Ht$ is symmetric and $L_H(t)\colon \Vt \to \Vt.$
We just write $L$ and $L(t)$ for the above restrictions. Furthermore, for almost every $t \in [0,T]$, assume
\begin{align}
\langle L(t)g,v\rangle_{\Vmt, \Vt} &= \langle g,L(t)v\rangle_{\Vmt,\Vt}&&\text{$\forall g \in \Vmt$, $\forall v \in \Vt$}\tag{L2}\label{eq:assSymmetricityOfL}\\
\norm{L(t)h}{\Ht} &\leq C_3 \norm{h}{\Ht}&&\text{$\forall h \in \Ht$}\label{eq:assBoundednessOfl}\tag{L3}\\
(L(t)h, h)_{\Ht} &\geq C_4\norm{h}{\Ht}^2&&\text{$\forall h \in \Ht$}\label{eq:assCoercivityOfl}\tag{L4}\\
Lv &\in L^2_V&&\text{$\forall v \in L^2_V$}\label{eq:assLvInL2V}\tag{L5}\\
v \in W(V,V^*) &\iff Lv \in W(V,V^*)\label{eq:assLvInW1}\tag{L6},\\
\intertext{and suppose the existence of a (linear symmetric) map $\dot{{L}}\colon L^2_V \to L^2_{V^*}$  (and identify $(\dot{L}v)(t)$ with $\dot{L}(t)v(t)$) satisfying}
\md (Lv) &= \dot{L} v + L \dot v \in L^2_{V^*} &&\text{$\forall v\in W(V,V^*)$}\label{eq:assProductRule}\tag{L7}\\
\lVert{\dot{L}(t)v}\rVert_{\Vmt} &\leq C_5\norm{v}{\Ht}&&\text{$\forall v \in \Vt$}\label{eq:assBoundednessOfdotL}\tag{L8}.
\end{align}
\end{asss}
\begin{asss}[Assumptions on $A(t)$]\label{asss:aWithoutL}
Suppose that the map 
\begin{align*}
\nonumber t &\mapsto \langle A(t)v(t),w(t)\rangle_{\Vmt, \Vt} \quad \forall v, w \in L^2_V
\end{align*}
{is measurable, and that there exist positive constants $C_1$, $C_2$ and $C_3$ independent of $t$ such that for almost every $t \in [0,T]$:}
\begin{align}
\langle A(t)v,v\rangle_{\Vmt, \Vt} &\geq C_1 \norm{v}{\Vt}^2-C_2 \norm{v}{\Ht}^2&&\forall v \in \Vt\tag{A1}\label{eq:assCoercivityOfa}\\
|\langle A(t)v,w\rangle_{\Vmt, \Vt}| &\leq C_3\norm{v}{\Vt}\norm{w}{\Vt}&&\forall v, w \in \Vt\tag{A2}\label{eq:assBoundednessOfa}.
\end{align}
\end{asss}
The standard equation $\dot u + A u + \symbolForBigC u = f$ is a special case of \eqref{eq:operatorEquation} when $L = \text{Id}$; in this case our demands in Assumptions \ref{asss:onL} are automatically met. 

\begin{theorem}[Well-posedness of \eqref{eq:operatorEquation}, {\cite[Theorem 3.6]{AlpEllSti}}]\label{thm:existenceWithoutL}
Under the assumptions in Assumptions \ref{asss:onL} and \ref{asss:aWithoutL}, for $f \in L^2_{V^*}$ and $u_0 \in H_0$, there is a unique solution $u \in W(V,V^*)$ satisfying \eqref{eq:operatorEquation} such that
\begin{equation*}
\norm{u}{W(V,V^*)} \leq C\left(\norm{u_0}{H_0}+\norm{f}{L^2_{V^*}}\right).
\end{equation*}
\end{theorem}
Now, suppose that  $f \in L^2_H$ and $u_0 \in V_0$. Under additional assumptions, we can obtain $\dot u \in L^2_H$. 
\begin{ass}\label{ass:basisFunctionsForRegularity}
It is assumed that there exists a basis $\{\chi_j^0\}_{j \in \mathbb{N}}$ of $V_0$ and a sequence $\{u_{0N}\}_{N \in \mathbb{N}}$ with $u_{0N} \in \text{span}\{\chi_1^0, ..., \chi_N^0\}$ for each $N$, such that
$u_{0N} \to u_0$ in $V_0$, $\norm{u_{0N}}{H_0} \leq C_1\norm{u_0}{H_0}$ and $\norm{u_{0N}}{V_0} \leq C_2\norm{u_0}{V_0}$,
where $C_1$ and $C_2$ do not depend on $N$ or $u_0$.
\end{ass}
\begin{remark}\label{rem:onBasisFunctions}
Thanks to Hilbert--Schmidt theory, such a basis as required by the last assumption always exists if $V_0 \subset H_0$ is compact. 
\end{remark}
Let us define the bilinear forms $l(t;\cdot,\cdot)\colon \Vmt \times \Vt \to \mathbb{R}$ and $a(t;\cdot,\cdot)\colon \Vt \times \Vt \to \mathbb{R}$ by $l(t;g,w) := \langle L(t) g,w \rangle_{\Vmt,\Vt}$ and 
$a(t;v,w) :=\langle A(t) v,w  \rangle_{\Vmt, \Vt}.$
For $\dot u$, $f \in L^2_{V^*}$, note that \eqref{eq:operatorEquation} is in fact equivalent to
\begin{equation*}\label{eq:defnWeakSolutionSecond}
\begin{aligned}
l(t; \dot u(t), v) + a(t;u(t),v) + \symbolForLittlec(t;u(t),v) &= \langle f(t), v \rangle_{\Vmt, \Vt}\\
u(0) &= u_0
\end{aligned}
\end{equation*}
for all $v \in V(t)$ and for almost every $t \in [0,T]$ (the null set is independent of $v$). A similar formulation holds if $\dot u$, $f \in L^2_H$. 
\begin{asss}[Further assumptions on $a(t;\cdot,\cdot)$]\label{asss:aAndasWithL}
Suppose that $a(t;\cdot,\cdot)$ has the form $a(t;\cdot,\cdot) = a_s(t;\cdot,\cdot) + a_n(t;\cdot,\cdot)$
where $a_s(t;\cdot, \cdot)\colon \Vt \times \Vt \to \mathbb{R}$ and $a_n(t;\cdot, \cdot)\colon\Vt \times \Ht \to \mathbb{R}$
are bilinear forms (we allow the possibility $a_n \equiv 0$) such that {the map} 
\begin{equation}\label{eq:assAbsCtyOfas}
t \mapsto a_s(t;y(t),y(t)) \text{ is absolutely continuous on $[0,T]$ for all $y \in \tilde C^1_V$.}\tag{A3}
\end{equation}
Suppose also that there exist positive constants $C_1$, $C_2$ and $C_3$ independent of $t$ such that
\begin{align}
|a_n(t;v,w)| &\leq C_1\norm{v}{\Vt}\norm{w}{\Ht}&&\forall v \in \Vt, w \in \Ht
\tag{A4}\label{eq:assBoundednessOfan}\\
|a_s(t;v,w)| &\leq C_2\norm{v}{\Vt}\norm{w}{\Vt}&&\forall v, w \in \Vt\tag{A5}\label{eq:assBoundednessOfas}\\
a_s(t;v,v) &\geq 0\tag{A6}&&\forall v \in \Vt \label{eq:assPositivityOfas}\\
\frac{d}{dt}a_s(t;y(t),y(t)) &= 2a_s(t;y(t),\dot y(t)) + r(t;y(t))&&\forall y \in \tilde{C}^1_V, \tag{A7}\label{eq:assDifferentiabilityOfas}\\
\intertext{for almost all $t \in [0,T]$, where the $\frac{d}{dt}$ here is the classical derivative, and $r(t;\cdot)\colon\Vt \to \mathbb{R}$ satisfies}
|r(t;v)| &\leq C_3\norm{v}{\Vt}^2&&\forall v \in \Vt\tag{A8}\label{eq:assBoundednessOfF}.
\end{align}
\end{asss}
\begin{remark}\label{rem:asssOna}
Note that we require only one part of the bilinear form $a(t;\cdot,\cdot)$ to be differentiable; however, any potentially non-differentiable terms require the stronger boundedness condition \eqref{eq:assBoundednessOfan}.
%
\end{remark}

\begin{theorem}[Regularity of the solution to \eqref{eq:operatorEquation}, {\cite[Theorem 3.13]{AlpEllSti}}]\label{thm:existenceWithL}
Under the assumptions in Assumptions \ref{asss:onL}, \ref{asss:aWithoutL}, \ref{ass:basisFunctionsForRegularity}, and \ref{asss:aAndasWithL}, if $f \in L^2_H$ and $u_0 \in V_0$, the unique solution $u$ of \eqref{eq:operatorEquation} from Theorem \ref{thm:existenceWithoutL} satisfies the regularity $u \in W(V,H)$ and the estimate
\begin{equation*}
\norm{u}{W(V,H)} \leq C\left(\norm{u_0}{V_0}+\norm{f}{L^2_H}\right).
\end{equation*}
\end{theorem}

%% file: arXivIFBApplications.tex
We now discuss evolving compact hypersurfaces (as defined in \S \ref{sec:evolvingHypersurfaces}) and evolving domains in the context of the abstract framework presented in \S \ref{sec:abstractFramework}.
\subsection{Evolving compact hypersurfaces}\label{sec:evolvingCompactHypersurfaces}
For each $t \in [0,T],$ let $\Gamma(t) \subset \mathbb{R}^{n+1}$ be a compact (i.e., no boundary) $n$-dimensional hypersurface of class $C^2$, and assume the existence of a flow $\Phi\colon [0,T] \times \mathbb{R}^{n+1} \to \mathbb{R}^{n+1}$ such that for all $t \in [0,T]$, with $\Gamma_0 := \Gamma(0)$, the map $\Phi_t^0(\cdot):=\Phi(t,\cdot)\colon \Gamma_0 \to \Gamma(t)$ is a $C^2$-diffeomorphism that satisfies
\begin{equation}\label{paramvel}
\begin{aligned}
\frac{d}{dt}\Phi^0_t(\cdot) &= \mathbf w(t,\Phi^0_t(\cdot))\\
\Phi^0_0(\cdot) &= \text{Id}(\cdot),
\end{aligned}
\end{equation}
where the map $\mathbf w\colon [0,T]\times\mathbb{R}^{n+1} \to \mathbb{R}^{n+1}$ is a velocity field (with normal component agreeing with the normal velocity of $\Gamma(t)$), and we assume that it is $C^2$ and satisfies the uniform bound
\[|\sgradt \cdot \mathbf w(t)| \leq C \qquad \text{for all $t \in [0,T]$.}\]
A 
normal vector field on the hypersurfaces is denoted by $\mathbf \nu\colon [0,T]\times \mathbb{R}^{n+1} \to \mathbb{R}^{n+1}$. Let $V(t) = H^1(\Gamma(t))$ and $H(t) = L^2(\Gamma(t)).$ We define the pullback operator by 
\[\phi_{-t} v = v \circ \Phi_t^0.\]
By \cite[Lemma 3.2]{vierling}, the map $\phi_{-t}$ is such that
\begin{align*}
&\phi_{-t}\colon L^2(\Gamma(t)) \to L^2(\Gamma_0)\qquad\text{and}\qquad
\phi_{-t}\colon H^1(\Gamma(t)) \to H^1(\Gamma_0)
\end{align*}
are linear homeomorphisms with the constants of continuity not dependent on $t$. We denote by $\phi_{-t}^*\colon H^{-1}(\Gamma_0) \to H^{-1}(\Gamma(t))$ the dual operator. The maps $t \mapsto \norm{\phi_{t}u}{X(t)}$ (for $X=L^2$ and $H^1$) are continuous \cite[Lemma 3.3]{vierling}, thus we have compatibility of the pairs $(H, \phi_{(\cdot)})$ and $(V, \phi_{(\cdot)}|_V)$, and the spaces $L^2_H = L^2_{L^2}$, $L^2_V = L^2_{H^1}$ and $L^2_{V^*} = L^2_{H^{-1}}$ are well-defined.

Let us now work out a formula for the strong material derivative. Note that, by the smoothness of 
$\Gamma(t)$, any function $u\colon \Gamma(t) \to \mathbb{R}$ can be extended to a neighbourhood of 
the space time surface $\cup _{t\in[0,T]}\Gamma(t) \times \{t\}$ in $\mathbb{R}^{n+2}$ in which  $\grad u$ and $u_{t}$ for the extension are well-defined ({see for example \cite[\S 2.2]{actanumerica}}).  The derivative of the pullback of a function $u \in C^1_V$ is
\begin{align*}
\frac{d}{dt}\phi_{-t}u(t) &= \frac{d}{dt}u(t,\Phi_t^0(y)) = u_t(t,\Phi_t^0(y)) + \grad u|_{(t,\Phi_t^0(y))}\cdot \mathbf w(t,\Phi_t^0(y))\\
&= \phi_{-t}u_t(t,y) + \phi_{-t}(\grad u(t,y))\cdot \phi_{-t}(\mathbf w(t,y)), \quad y \in \Gamma_{0}
\end{align*}
giving $\dot u(t,x) = u_t(t,x) + \grad u(t,x)\cdot \mathbf w(t,x)$  for $x\in \Gamma(t).$ The expression on the right hand side is independent of the extension. It is clear that our definition of the strong material derivative coincides with the well-established definition \eqref{eq:mdRough}. 

We denote by $J_t^0$ the change of area element when transforming from
$\Gamma_0$ to $\Gamma(t)$, i.e., for any integrable function $\zeta\colon
\Gamma(t) \to \mathbb R$
$$ \int_{\Gamma(t)} \zeta = \int_{\Gamma_0} (\zeta \circ \Phi_t^0) J_t^0
= \int_{\Gamma_0} \phi_{-t} \zeta J_t^0.$$
Using the transport identity
$$ \frac{d}{dt} \int_{G(t)} \zeta(t) \Big{|}_{t} = \int_{G(t)}
\dot{\zeta}(t) + \zeta(t) \nabla_{G(t)} \cdot \mathbf {w}(t) $$
on any portion $G \subset \Gamma$ with points that move with the
velocity field $\mathbf {w}$ (for instance, see \cite{dziuk_elliott}) one can
easily show that
\begin{equation}\label{eq:derivativeOfJacobian}
\frac{d}{dt}J_t^0=\phi_{-t}(\sgradt \cdot \mathbf{w}(t))J_t^0.
\end{equation}
The field $J_t^0$ is uniformly bounded by positive constants 
\[\frac{1}{C_J} \leq J_t^0(z) \leq C_J \qquad \text{for all $z \in \Gamma_0$ and for all $t \in [0,T]$}.\]
The $L^2(\Gamma(t))$ inner product is
\begin{align*}
(u ,v )_{L^2(\Gamma(t) )} &= \int_{\Gamma(t)}u v  = \int_{\Gamma_0}\phi_{-t}u \phi_{-t}v J_t^0.
\end{align*}
The bilinear form $\hat{b}(t;\cdot,\cdot)\colon H_0\times H_0 \to \mathbb{R}$ (defined by $(u,v)_{\Ht} = \hat{b}(\phi_{-t} v, \phi_{-t} v)$) is 
\[\hat{b}(t;u_0,v_0) = \int_{\Gamma_0}u_0v_0J_t^0,\]
so the action of the operator $T_t\colon H_0 \to H_0$ (see Definition \ref{defn:bilinearFormb} and Theorem \ref{thm:spaceW1}) is just pointwise multiplication:
\[T_t u_0 = J_t^0 u_0.\]
We see that the function $\theta$ from Assumptions \ref{asss:differentiabilityOfNorm} is
\begin{align*}
\theta(t,u_0) &= \frac{d}{dt}\norm{\phi_t u_0}{L^2(\Gamma(t))}^2  = \frac{d}{dt}\int_{\Gamma_0}u_0^2 J_t^0 = \int_{\Gamma_0}u_0^2 \phi_{-t}(\sgradt \cdot \mathbf w(t))J_t^0 =\int_{\Gamma(t)}(\phi_tu_0)^2\sgrad \cdot \mathbf w(t),
\end{align*}
where the cancellation of the Jacobian terms in the last equality is due to the inverse function theorem. Now, $v \mapsto \theta(t,v)$ is continuous because if $v_n \to v$ in $L^2(\Gamma_0)$, then $v_n^2 \to v^2$ in $L^1(\Gamma_0)$ and so
\begin{align*}
|\theta(t,v_n) - \theta(t,v)| &\leq \int_{\Gamma_0}|v_n^2 - v^2||\phi_{-t}(\sgradt \cdot \mathbf{w}(t))J_t^0| \leq C\norm{v_n^2 - v^2}{L^1(\Gamma_0)} \to 0.
\end{align*}
Finally, \begin{align*}
|\theta(t, u_0+v_0) - \theta(t, u_0 - v_0)| &= \bigg|4\int_{\Gamma(t)}\phi_t u_0 \phi_t v_0\sgradt \cdot \mathbf{w}(t)\bigg| \leq C\norm{u_0}{L^2(\Gamma_0)}\norm{v_0}{L^2(\Gamma_0)}.
\end{align*}
So we have checked Assumptions \ref{asss:differentiabilityOfNorm}. Now if $u_0$, $v_0 \in L^2(\Gamma_0)$,
\begin{align*}
\hat{\symbolForLittlec}(t;u_0,v_0) &= \frac{\partial}{\partial t}\hat{b}(t;u_0,v_0) = \int_{\Gamma_0}u_0v_0\phi_{-t}(\sgradt \cdot \mathbf{w})J_t^0,
\end{align*}
thus the bilinear form $\symbolForLittlec(t;\cdot,\cdot)$ of Definition \ref{defn:bilinearFormc} is 
\begin{align*}
\symbolForLittlec(t;u,v) &= \int_{\Gamma_0}\phi_{-t}u\phi_{-t}v\phi_{-t}(\sgradt \cdot \mathbf{w})J_t^0 = \int_{\Gamma(t)}uv\sgradt \cdot \mathbf{w},
\end{align*}
which, as claimed, is measurable in $t$ and bounded on $\Ht \times \Ht$.
So then $u \in L^2_V$ has a weak material derivative $\dot u \in L^2_{V^*}$ if and only if
\[\int_0^T \langle \dot u(t), \eta(t) \rangle_{\Vmt,\Vt} = - \int_0^T \int_{\Gamma(t)}u(t)\dot \eta(t) - \int_0^T \int_{\Gamma(t)}u(t)\eta(t)\sgradt \cdot \mathbf w(t)\]
holds for all $\eta \in \mathcal{D}_V(0,T)$ (cf. \cite{vierling, reusken}).

Finally, \cite[Lemma 3.7]{vierling} proves that $T_{(\cdot)}u(\cdot) \in \mathcal W(V_0,V_0^*)$ if and only if $u \in \mathcal W(V_0,V_0^*)$, due to the fact that both $J_{(\cdot)}^0$ and its reciprocal $1/J_{(\cdot)}^0$ are in $C^1([0,T]\times \Gamma_0).$ To see that the evolving space equivalence (Assumption \ref{ass:spaceW1}) holds, take $u \in \mathcal W(V_0,V_0^*)$ and obtain by the product rule and \eqref{eq:derivativeOfJacobian} the identity
\[(J_t^0u(t))' = J_t^0u'(t) + \phi_{-t}(\sgradt \cdot \mathbf w)J_t^0u(t).\]
Therefore, the maps $\hat S(t)$ and $\hat D(t)$  (from Theorem \ref{thm:spaceW1}) are $\hat S(t)u'(t) = J_t^0u'(t)$ and $\hat D(t) \equiv 0$. It follows by the smoothness of $\Phi^0_t$ and $J_t^0$ that $\hat S(\cdot)u'(\cdot) \in L^2(0,T;V_0^*)$. By Theorem \ref{thm:spaceW1}, we have that the space $W(V,V^*) = \{ u \in L^2_{H^1} \mid \dot u \in L^2_{H^{-1}}\}$ is indeed isomorphic to $\mathcal W(V_0, V_0^*)$ and there is an equivalence of norms between
\[\norm{u}{W(V,V^*)} \qquad\text{and}\qquad \norm{\phi_{-(\cdot)}u(\cdot)}{\mathcal W(V_0, V_0^*)}.\]
See also \cite[Lemma 3.9]{vierling}. It is easy to see that $W(V,H)$ and $\mathcal W(V_0, H_0)$ are also equivalent.

\subsection{Evolving domains}\label{sec:evolvingFlatHypersurfaces}
We discuss here what is common to the three examples on evolving domains and leave the specifics and peculiarities to be detailed on a case-by-case basis as required.

For each $t\in [0,T],$ let $\Omega(t) \subset \mathbb R^{n}$ be a bounded open and connected domain of class $C^2$ with boundary $\Gamma(t)$. It is possible to view $\Omega(t)$ as an evolving flat hypersurface in $\mathbb R^{n+1}$ (see Remark \ref{rem:evolvingDomainsAreFlatHypersurfaces}), though we choose not to follow this approach. The boundary $\Gamma(t)$ is an evolving compact $(n-1)$-dimensional hypersurface in $\mathbb R^{n}$. We denote $\Omega_0 := \Omega(0)$ and $\Gamma_0 := \Gamma(0)$. For each $t \in [0,T]$, we assume the existence of a map $\Phi^0_t\colon \overline{\Omega}_0 \to \overline{\Omega(t)}$ such that $\Phi^0_t(\Omega_0) = \Omega(t),$ $\Phi^0_t(\Gamma_0) = \Gamma(t)$,
\begin{align*}
\Phi^0_t\colon \Omega_0 \to \Omega(t)\text{ is a $C^2$-diffeomorphism} \qquad \text{and} \qquad \Phi^0_{(\cdot)} \in C^2([0,T]\times \overline{\Omega_0}).
\end{align*}
We assume that $\Phi^0_t$ satisfies the ODE \eqref{paramvel} on $\overline{\Omega_0}$ for a $C^2$ velocity $\mathbf w$ (with the normal part of $\mathbf w$ agreeing with the normal velocity of the domain) with $|\grad \cdot \mathbf w(t)|$ and $|\sgradt \cdot \mathbf w(t)|$ both bounded above uniformly in $t$, like before. We write $\Phi^t_0 := (\Phi^0_t)^{-1}$.
%
\begin{defn}\label{defn:mapsPhi}
For functions $u\colon \Omega_0 \to \mathbb{R}$ and $v\colon \Gamma_0 \to \mathbb{R}$, define the restrictions
\begin{align*}
\phi_{\Omega,t}u = u \circ \Phi^t_0|_{\Omega_0} \qquad \text{and}\qquad \phi_{\Gamma,t}v = v \circ \Phi^t_0|_{\Gamma_0}.
\end{align*}
\end{defn}
We find that 
\begin{align*}
&\phi_{\Omega,t} \colon H^1(\Omega_0) \to H^1(\Omega(t))\qquad\text{and}\qquad\phi_{\Omega,t} \colon L^2(\Omega_0) \to L^2(\Omega(t))
\end{align*}
are linear homeomorphisms with the constants of continuity not depending on $t$
(we can either adapt the proofs in \cite{vierling} or use Problem 1.3.1 in \cite{nobile}). One of the most important terms in the solution space regime is the Jacobian $J_{\Omega,(\cdot)}^0 := \text{det}\mathbf{D}\Phi^0_{(\cdot)} \in C^1([0,T]\times \Omega_0)$; one can show that it satisfies much of the same properties (see \cite{bonito} for this) as the Jacobian term did in \S \ref{sec:evolvingCompactHypersurfaces} for the case of compact hypersurfaces. Hence it is straightforward to adapt the proofs for the case of a domain with boundary to yield the fulfilment of the evolving space equivalence Assumption \ref{ass:evolvingSpaceEquivalenceStronger} between $\mathcal{W}(H^1(\Omega_0),(H^1(\Omega_0))^*)$ and $W(H^1_\Omega,(H^1_\Omega)^*)$, and $\mathcal{W}(H^1(\Omega_0),L^2(\Omega_0))$ and $W(H^1_\Omega,L^2_\Omega)$.

Furthermore, assuming
\[\Phi^0_t\colon \Gamma_0 \to \Gamma(t)\text{ is a $C^2$-diffeomorphism,}
\]
since the boundary $\Gamma(t)$ is a $C^2$ hypersurface, it satisfies the assumptions in \S \ref{sec:evolvingCompactHypersurfaces} and so it follows that the maps
\begin{align*}
&\phi_{\Gamma,t} \colon H^1(\Gamma_0) \to H^1(\Gamma(t))\qquad\text{and}\qquad\phi_{\Gamma,t} \colon L^2(\Gamma_0) \to L^2(\Gamma(t))
\end{align*}
are also linear homeomorphisms with the constants of continuity not depending on $t$. 
The trace map $\tau_t\colon H^1(\Omega(t)) \to L^2(\Gamma(t))$ (see  \cite[\S I.8, Theorem 8.7]{wloka}) will play a prominent role. We need the following lemma to show that the constant in the trace inequality is uniform in time.
\begin{lem}\label{lem:traceCommutesWithPhiL2}
For all $w \in H^1(\Omega_0)$, the equality $\tau_t(\phi_{\Omega,t} w) = \phi_{\Gamma,t}(\tau_0w)$ holds in $L^{2}(\Gamma(t))$. 
\end{lem}
\begin{proof}
This is because $\tau_t(\phi_{\Omega,t} w_n)  = \phi_{\Gamma,t}(\tau_0w_n)$ holds for all $w_n \in C^1(\overline{\Omega_0})$ (one can see this identity by using the fact that the same formula defines $\phi_{\Omega,t}$ and $\phi_{\Gamma,t}$ and that $\Phi^t_0$ maps boundary to boundary), in particular, it holds for $w_n \in C^1(\overline{\Omega_0}) \cap H^1(\Omega_0)$ such that $w_n \to w$  in $H^1(\Omega_0)$. Then by continuity of the various maps, we can pass to the limit and obtain the identity.
\end{proof}
Now let $u \in H^1(\Omega_0)$. Using Lemma \ref{lem:traceCommutesWithPhiL2} and the properties of the maps $\phi_{\Gamma,t}$ and $\phi_{\Omega,t}$, we obtain
\[\norm{\tau_0 u}{L^2(\Gamma_0)} \geq C_1\norm{\phi_{\Gamma, t}(\tau_0 u)}{L^2(\Gamma(t))} = C_1\norm{\tau_t(\phi_{\Omega, t}u)}{L^2(\Gamma(t))}\]
{and}
\[\norm{u}{H^1(\Omega_0)} \leq C_2 \norm{\phi_{\Omega, t} u}{H^1(\Omega(t))},\]
{and these inequalities together with the trace inequality on $\Omega_0$ 
imply the existence of $C_T$ such that} 
\begin{equation}\label{eq:bs_trace}
\norm{\tau_t u}{L^2(\Gamma(t))} \leq C_T\norm{u}{H^1(\Omega(t))}\quad \forall u \in H^1(\Omega(t)), \forall t \in [0,T].
\end{equation}
\begin{rem}\label{rem:wIsImportant}Observe that the velocity field $\mathbf w$ may have no physical or actual relevance to a particular problem posed on an evolving hypersurface apart from having the normal component of $\mathbf w$ agreeing with the normal velocity of the hypersurface (or domain). The tangential component of $\mathbf w$ can be chosen arbitrarily, as mentioned before. On the other hand, $\mathbf w$ plays an indispensable role in the definition of the function spaces in which we look for solutions.
\end{rem}
\section{Weak formulation and well-posedness}\label{sec:weakFormulationAndWellPosedness}
We are now in a position to prove the well-posedness of the equations in \S \ref{sec:pdes} in a weak sense.
\subsection{The surface advection-diffusion equation \eqref{eq:surfaceHeatEquationStart}}\label{sec:she}
Let us assume for simplicity that $\mathbf b = \mathbf w$ in \eqref{eq:surfaceHeatEquationStart}; that is, the physical velocity agrees with the velocity of the parametrisation. Let us suppose that $\Gamma(t)$ possesses the properties in \S \ref{sec:evolvingCompactHypersurfaces}. Availing ourselves of the framework in \S \ref{sec:evolvingCompactHypersurfaces}, the weak formulation of \eqref{eq:surfaceHeatEquationStart} asks to find $u \in W(V,V^*)$ such that
\[\int_0^T\langle \dot u(t), v(t) \rangle_{H^{-1}(\Gamma(t)),H^{1}(\Gamma(t))} + \int_0^T\int_{\Gamma(t)}\sgrad u(t)\cdot \sgrad v(t) + \int_0^T\int_{\Gamma(t)}u(t)v(t)\sgrad \cdot \mathbf w(t) = 0\]
holds for all $v \in L^2_V$. Here,
\begin{align*}
a(t;u,v) = \int_{\Gamma(t)}\sgrad u\cdot \sgrad v
\end{align*}
which clearly satisfies the assumptions listed in Assumptions \ref{asss:aWithoutL}. 
Applying Theorem \ref{thm:existenceWithoutL}, we obtain a unique solution $u \in W(V,V^*)$. If instead we ask for $\dot u \in L^2_H$, in addition to requiring $u_0 \in H^1(\Gamma_0)$, we need to check Assumptions \ref{ass:basisFunctionsForRegularity} and \ref{asss:aAndasWithL}; {the former follows since for example we can take $\chi_j^0$ to be the eigenfunctions of the Laplacian (see Remark \ref{rem:onBasisFunctions})}. We take $a_s \equiv a$ as defined above and set $a_n \equiv 0$. Most of the remaining assumptions are easy to check. {For \eqref{eq:assAbsCtyOfas}, we see from \cite[Lemma 2.2]{dziuk_elliott} that for $\eta \in {C}^\infty_V$, the pointwise derivative}
\begin{align*}
\frac{d}{dt}\int_{\Gamma(t)} |\sgrad \eta(t)|^2 &= \int_{\Gamma(t)}(2\sgrad \eta(t) \cdot \sgrad \dot{\eta}(t) - 2\sgrad \eta(t)(\mathbf{D}_\Gamma \mathbf w(t))\sgrad \eta(t) + |\sgrad \eta(t)|^2 \sgrad \cdot \mathbf w(t))
\end{align*}
{holds everywhere with $(\mathbf{D}_\Gamma \mathbf w(t))_{ij} := \underline{D}_j\mathbf w^i(t).$ Since the right hand side of the above expression is in $L^1(0,T)$, we have that the derivative is in fact a weak derivative. By a density argument, we find that the formula above holds in the weak sense also for $\eta \in \tilde C^1_V$.}
{ Since the right hand side and the term being differentiated on the left hand side are in $L^1(0,T)$, it follows that $t \mapsto \int_{\Gamma(t)} |\sgrad \eta(t)|^2$ has an absolutely continuous representative with the pointwise a.e. derivative as above, giving \eqref{eq:assDifferentiabilityOfas}.}
It is easy to see that
\[r(t;\eta) = \int_{\Gamma(t)}(- 2\sgrad \eta(\mathbf{D}_\Gamma \mathbf w(t))\sgrad \eta + |\sgrad \eta|^2 \sgrad \cdot \mathbf w(t))\]
satisfies \eqref{eq:assBoundednessOfF}. Finally, an application of Theorem \ref{thm:existenceWithL} 
shows that $u \in W(V,H)$.
\begin{rem}
We mentioned in Remark \ref{rem:differentVelocityFields} that if $\mathbf w$ is purely tangential, the surface does not evolve. However, even in this situation, it can still be useful to think of spaces of functions on $\Gamma(t) \equiv \Gamma_0$ as $H(t)$ and $V(t)$ (i.e., still parametrised by $t \in [0,T]$). Consider the surface heat equation
\[ \dot u - \slap u + u \sgrad \cdot \mathbf w = f.\]
{If $\mathbf w(t,\cdot)$ is a tangential velocity field, then this equation corresponds to}
\[u_t - \slap u +u \sgrad \cdot \mathbf w+ \mathbf w\cdot \sgrad u = f,\]
which could be advection-dominated (if $\mathbf w$ is sufficiently large) and potentially problematic for numerical computations. The first formulation, in which we make use of $H(t)$ and $V(t)$ for each $t \in [0,T]$, avoids this issue.
\end{rem}

\subsection{The bulk equation \eqref{eq:bulkEquationStart}}\label{sec:bulkequation}
Here, we use the notations and results of \S \ref{sec:evolvingFlatHypersurfaces}.  
Observe that the velocity field $\mathbf w$ does not appear in the physical equation \eqref{eq:bulkEquationStart}; $\mathbf w$ is an extension to the interior (or bulk) of the boundary velocity, and the normal component of this boundary velocity must agree with the normal velocity of $\Omega(t)$. For example, if the normal velocity of $\Omega(t)$ were $\mathbf b \cdot \nu$ then $\mathbf w$ can be taken to be an extension of $\mathbf b \cdot \nu$. In this sense, $\mathbf w$ is not relevant to the physical problem but it is essential to the functional setting we have built up (see Remark \ref{rem:wIsImportant}). Let $V(t) = H_0^1(\Omega(t))$ and $H(t) = L^2(\Omega(t))$. With $\phi_{t}$ referring to the map $\phi_{\Omega, t}$ from Definition \ref{defn:mapsPhi}, it follows from \S \ref{sec:evolvingFlatHypersurfaces} that $(H, \phi_{(\cdot)})$ and $(V, \phi_{(\cdot)}|_V)$ are compatible and that there is an evolving space equivalence between $\mathcal{W}(V_0,V_0^*)$ and $W(V,V^*).$ For convenience, set $\mathbf{p}:= \mathbf b-\mathbf{w}$. Our weak formulation is: with $f \in L^2_H$ and $u_0 \in \Vs$, find $u \in W(V,H)$ such that
\begin{align*}
\int_0^T\int_{\Omega(t)}(\dot u(t)v(t) + \mathbf p(t)\cdot \grad u(t)v(t)+\grad \cdot \mathbf b(t)u(t)v(t) +D\grad u(t)\cdot \grad v(t))&= \int_0^T\int_{\Omega(t)}f(t)v(t)\\
u(0) &= u_0
\end{align*}
holds for all $v \in L^2_V$. {Now, Assumption \ref{ass:basisFunctionsForRegularity} holds just like in the previous example.} We need to check Assumptions \ref{asss:aWithoutL} and \ref{asss:aAndasWithL}. We have
\begin{align*}
a(t;u,v) &=\int_{\Omega(t)}\mathbf{p}(t)\cdot \grad uv + (\grad \cdot \mathbf b(t))uv+ D\grad u\cdot\grad v
\end{align*}
with
\begin{align*}
a_s(t;u,v) = \int_{\Omega(t)} D\grad u\cdot\grad v \;\; \text{and}\;\; a_n(t;u,v) = \int_{\Omega(t)}((\grad \cdot \mathbf b(t))u+\mathbf{p}(t)\cdot \grad u)v.
\end{align*}
The boundedness of $a(t;\cdot,\cdot)$ is easy, while coercivity can be shown by the use of Young's equality with $\epsilon$:
\begin{align*}
a(t;v,v) 
&\geq  D\norm{\grad v}{L^2(\Omega(t))}^2-\frac{C}{2D}\norm{\mathbf{p}^2(t)}{L^\infty(\Omega(t))}\norm{v}{L^2(\Omega(t))}^2 - \frac{D}{2}\norm{\grad v}{L^2(\Omega(t))}^2\\
&\quad -\norm{\grad \cdot \mathbf b(t)}{L^\infty(\Omega(t))}\norm{v}{L^2(\Omega(t))}^2\\
&= -\left(\frac{C}{2D}\norm{\mathbf{p}^2(t)}{L^\infty(\Omega(t))}+\norm{\grad \cdot \mathbf b(t)}{L^\infty(\Omega(t))}\right)\norm{v}{L^2(\Omega(t))}^2+\frac{D}{2}\norm{\grad v}{L^2(\Omega(t))}^2.
\end{align*}
Coming to the term $a_s(t;\cdot,\cdot)$; firstly, positivity and boundedness are obvious, and {absolute continuity} and a.e. differentiability are the same as for the bilinear form $a(t;\cdot,\cdot)$ in the previous example:
$$
\frac{d}{dt}a_s(t;\eta(t),\eta(t))= 2a_s(t;\dot \eta(t), \eta(t)) +r(t;\eta(t))
$$
for $\eta \in \tilde{C}^1_V$, where 
\begin{align*}
r(t;\eta(t)) &=D\int_{\Omega(t)}(- 2\grad \eta(t)(\mathbf{D} \mathbf w(t))\grad \eta(t) + |\grad \eta(t)|^2 \grad \cdot \mathbf w(t))
\end{align*}
which is obviously bounded. Finally, the uniform bound on $a_n(t;\cdot,\cdot)\colon \Vt \times \Ht \to \mathbb{R}$
follows by the assumptions made on $\mathbf b$ in \S \ref{sec:pdes}. With all the assumptions checked, we apply Theorem \ref{thm:existenceWithL} and find a unique solution $u \in W(V,H)$.
\subsection{The coupled bulk-surface system \eqref{eq:bs_1}--\eqref{eq:bs_ICforv}}\label{sec:cbss}
We are again going to use the framework of \S \ref{sec:evolvingFlatHypersurfaces}. The setting up of the function spaces is slightly more involved now.
%
\subsubsection{Function spaces}
Define the product Hilbert spaces
\[V(t) = H^1(\Omega(t)) \times H^1(\Gamma(t))\qquad\text{and}\qquad H(t) = L^2(\Omega(t)) \times L^2(\Gamma(t))\]
which we equip with the inner products
\begin{align*}
((\omega_1, \gamma_1),(\omega_2, \gamma_2))_{H(t)} &= (\omega_1,\omega_2)_{L^2(\Omega(t))} + (\gamma_1, \gamma_2)_{L^2(\Gamma(t))}\\
((\omega_1, \gamma_1),(\omega_2, \gamma_2))_{V(t)} &= (\omega_1,\omega_2)_{H^1(\Omega(t))} + (\gamma_1, \gamma_2)_{H^1(\Gamma(t))}.
\end{align*}
Clearly $\Vt \subset \Ht$ is continuous and dense and both spaces are separable. The dual space of $V(t)$ is $\Vmt = (H^{1}(\Omega(t)))^* \times H^{-1}(\Gamma(t))$ and the duality pairing is
\begin{align*}
\langle (f_\omega, f_\gamma), (u_\omega, u_\gamma) \rangle_{\Vmt,\Vt} &= \langle f_\omega, u_\omega \rangle_{(H^{1}(\Omega(t)))^*,H^{1}(\Omega(t))}+ \langle f_\gamma, u_\gamma \rangle_{H^{-1}(\Gamma(t)),H^{1}(\Gamma(t))}.
\end{align*}
Define the map $\phi_{t}\colon H_0 \to H(t)$ by 
\[\phi_{t}((\omega, \gamma)) = (\phi_{\Omega,t}\omega, \phi_{\Gamma,t}\gamma)\]
where $\phi_{\Omega,t}$ and $\phi_{\Gamma,t}$ are as defined previously.
From \S \ref{sec:evolvingCompactHypersurfaces} and \S \ref{sec:evolvingFlatHypersurfaces}, we find that $(H, \phi_{(\cdot)})$ and $(V, \phi_{(\cdot)}|_V)$ are compatible, and we have the evolving space equivalence between $\mathcal{W}(V_0,V_0^*)$ and $W(V,V^*).$

To define the weak material derivative, note that because the inner product on $H(t)$ 
is a sum of the $L^2$ inner products on $\Omega(t)$ and $\Gamma(t)$, it follows that the bilinear form $\symbolForLittlec(t;\cdot,\cdot)$  is
\[\symbolForLittlec(t;(\omega_1, \gamma_1),(\omega_2,\gamma_2))  = \symbolForLittlec_\Omega(t;\omega_1, \omega_2) + \symbolForLittlec_\Gamma(t;\gamma_1, \gamma_2)\]
with
\[\symbolForLittlec_\Omega(t;\omega_1, \omega_2) = \int_{\Omega(t)}\omega_1 \omega_2 \grad_\Omega \cdot \mathbf w(t)\quad \text{and}\quad \symbolForLittlec_\Gamma(t;\gamma_1, \gamma_2) = \int_{\Gamma(t)}\gamma_1 \gamma_2 \grad_\Gamma \cdot \mathbf w(t)\]
being the bilinear forms associated with the material derivatives of the constituent spaces of the product space.

\subsubsection{Weak formulation and well-posedness}

To obtain the weak form, we let $(\omega, \gamma) \in L^2_V$ and take the inner product of \eqref{eq:bs_1} with $\omega$ and the inner product of \eqref{eq:bs_2} with $\gamma$:
\begin{align}
\int_{\Omega(t)}\dot u\omega + \int_{\Omega(t)} \grad_\Omega u\cdot \grad_\Omega \omega - \int_{\Gamma(t)}\omega \grad_\Omega u \cdot \nu + \int_{\Omega(t)}u\omega \grad_\Omega \cdot \mathbf w&= \int_{\Omega(t)}f\omega\label{eq:bs_wf21}\\
\int_{\Gamma(t)}\dot v\gamma + \int_{\Gamma(t)}\sgrad v \cdot \sgrad \gamma + \int_{\Gamma(t)}v\gamma\sgrad \cdot \mathbf w + \int_{\Gamma(t)}\gamma \grad_\Omega u \cdot \nu  &= \int_{\Gamma(t)}g\gamma.\label{eq:bs_wf22}
\end{align}
Multiplying \eqref{eq:bs_wf21} by $\alpha$ and \eqref{eq:bs_wf22} by $\beta$, taking the sum and substituting the boundary condition \eqref{eq:bs_bc}, we end up with
\begin{align*}
\nonumber &\alpha\int_{\Omega(t)}\dot u \omega +\beta \int_{\Gamma(t)}\dot v \gamma
 + \alpha\int_{\Omega(t)}\grad_\Omega u\cdot \grad_\Omega \omega+\beta \int_{\Gamma(t)}\sgrad v \cdot \sgrad \gamma+ \alpha\int_{\Omega(t)}u\omega \grad_\Omega \cdot \mathbf w\\
\nonumber & +\beta \int_{\Gamma(t)}v\gamma\sgrad \cdot \mathbf w+ \int_{\Gamma(t)}(\beta v - \alpha u)(\beta \gamma -\alpha\omega) = \alpha\int_{\Omega(t)}f\omega +\beta \int_{\Gamma(t)}g\gamma.
\end{align*}
Defining the bilinear forms
\begin{align*}
l(t;(\dot u, \dot v),(\omega,\gamma)) &= \alpha\langle \dot u, \omega\rangle_{ (H^{1}(\Omega(t)))^*,H^{1}(\Omega(t))} + \beta \langle \dot v, \gamma\rangle _{H^{-1}(\Gamma(t)),H^{1}(\Gamma(t))}\\
a(t;(u,v),(\omega,\gamma)) &= \alpha\int_{\Omega(t)}\grad_\Omega u\cdot \grad_\Omega \omega +\beta \int_{\Gamma(t)}\sgrad v \cdot \sgrad \gamma + \int_{\Gamma(t)}(\beta v - \alpha u)(\beta \gamma -\alpha\omega),
\end{align*}
our weak formulation reads: given $(f,g) \in L^2_H$ and $(u_0, v_0) \in \Vs$, find $(u,v) \in W(V,H)$ such that 
\begin{equation}\label{eq:coupledbulksurface}
\begin{aligned}
\nonumber \int_0^T (l(t;(\dot u,\dot v),(\omega,\gamma)) + a(t;(u,v),(\omega,\gamma)) + \symbolForLittlec(t;(u,v),(\omega,\gamma))) &= \int_0^T((\alpha f, \alpha g),(\omega,\gamma))_{\Ht}\\
(u(0),v(0)) &= (u_0,v_0)
\end{aligned}\tag{$\textbf{P}_{\textbf{bs}}$}
\end{equation}
for all $(\omega, \gamma) \in L^2_V$. Note that Assumption \ref{ass:basisFunctionsForRegularity} holds due to the compactness of $V_0 \subset H_0$. Let us now check Assumptions \ref{asss:onL}.
\paragraph{Assumptions \eqref{eq:assLvstarInL2Vstar}--\eqref{eq:assBoundednessOfdotL}}We can write
\[ l(t;(\dot u,\dot v),(\omega,\gamma))=\langle L(t)(\dot u, \dot v), (\omega, \gamma) \rangle_{\Vmt,\Vt} = \langle (\alpha \dot u, \beta \dot v), (\omega, \gamma) \rangle_{\Vmt, \Vt},\]
i.e., $L(\dot u, \dot v)$ is the functional $\int_0^T\langle (\alpha \dot u(t), \beta \dot v(t)), (\cdot)(t) \rangle_{\Vmt, \Vt}$, which obviously satisfies \eqref{eq:assLvstarInL2Vstar}. We see that $L\colon L^2_H \to L^2_H$, and when $(\dot u,\dot v) \in H(t),$
\[\langle L(t)(\dot u,\dot v), (\omega, \gamma)\rangle = ((\alpha \dot u, \beta \dot v),(\omega, \gamma))_{\Ht},\]
so indeed $L(t)|_{\Ht}$ has range in $\Ht$ and $L(t)|_{\Vt}$ has range in $\Vt$.  
Assumptions \eqref{eq:assSymmetricityOfL}--\eqref{eq:assLvInL2V} are immediate, and \eqref{eq:assLvInW1} also follows easily. For \eqref{eq:assProductRule} and \eqref{eq:assBoundednessOfdotL}, note that the map $\dot{L} \equiv 0$.

We also need to check Assumptions \ref{asss:aWithoutL} and \ref{asss:aAndasWithL} on the bilinear form $a(t;\cdot,\cdot)$. Set $\mathbf{v_i} = (\omega_i, \gamma_i)$ for $i = 1, 2.$ 
Coercivity of $a(t;\cdot,\cdot)$ (assumption \eqref{eq:assCoercivityOfa}) is achieved with no great difficulty (one uses the $L^\infty$ bound on $\mathbf{w}\cdot \mu$, the trace inequality and Young's inequality with $\epsilon$).
\paragraph*{Assumption \eqref{eq:assBoundednessOfa}}
For boundedness of $a(t;\cdot,\cdot)$, we start with
\begin{align}
|a(t;\mathbf{v_1}, \mathbf{v_2})| &\leq C\norm{\mathbf{v_1}}{\Vt}\norm{\mathbf{v_2}}{\Vt} + \int_{\Gamma(t)}|\beta^2 \gamma_1\gamma_2 + \alpha^2 \omega_1\omega_2 -\alpha\beta(\omega_1\gamma_2 +\gamma_1\omega_2)|.
\label{eq:checkAssumptionBoundedness}
\end{align}
The trace inequality \eqref{eq:bs_trace} allows us to estimate the last term of \eqref{eq:checkAssumptionBoundedness} as follows:
\begin{align*}
\int_{\Gamma(t)}&|\beta^2 \gamma_1\gamma_2 + \alpha^2 \omega_1\omega_2 -\alpha \beta(\omega_1\gamma_2 + \gamma_1\omega_2)|\\
&\leq \beta^2\norm{\gamma_1}{L^2(\Gamma(t))}\norm{\gamma_2}{L^2(\Gamma(t))} + \alpha^2 C_T^2\norm{\omega_1}{H^1(\Omega(t))}\norm{\omega_2}{H^1(\Omega(t))}\\
&\quad +  \alpha \beta C_T\left(\norm{\omega_1}{H^1(\Omega(t))}\norm{\gamma_2}{L^2(\Gamma(t))}+ \norm{\gamma_1}{L^2(\Gamma(t))}\norm{\omega_2}{H^1(\Omega(t))}\right)\\
&\leq C\norm{(\omega_1,\gamma_1)}{V(t)}\norm{(\omega_2,\gamma_2)}{V(t)}= C\norm{\mathbf{v_1}}{V(t)}\norm{\mathbf{v_2}}{V(t)}.
\end{align*}
\paragraph*{Assumptions \eqref{eq:assDifferentiabilityOfas} and \eqref{eq:assBoundednessOfF}}We do not require the splitting of $a(t;\cdot,\cdot)$ into a differentiable and non-differentiable part since $a(t;\cdot,\cdot)$ is differentiable as shown below ({the absolute continuity follows like before}). In view of this and Remark \ref{rem:asssOna}, we still need to check \eqref{eq:assDifferentiabilityOfas} and \eqref{eq:assBoundednessOfF}. Let us define
\begin{align*}
a_{\Omega}(t;\omega_1,\omega_2) = \alpha\int_{\Omega(t)}\grad_\Omega \omega_1 \cdot \grad_\Omega \omega_2 \qquad \text{and} \qquad
a_{\Gamma}(t;\gamma_1,\gamma_2) = \beta\int_{\Gamma(t)}\sgrad \gamma_1 \cdot \sgrad \gamma_2,
\end{align*}
so that
\begin{align*}
a(t;(\omega_1,\gamma_1),(\omega_2,\gamma_2)) &= a_{\Omega}(t;\omega_1,\omega_2) + a_{\Gamma}(t;\gamma_1,\gamma_2) + \int_{\Gamma(t)}(\beta \gamma_1 - \alpha \omega_1)(\beta \gamma_2 -\alpha\omega_2)
\end{align*}
Taking $\mathbf{v_1} \in \tilde{C}^1_V,$ we differentiate:
\begin{align*}
\frac{d}{dt}a(t;\mathbf{v_1}, \mathbf{v_1}) 
&= 2a_{\Omega}(t;\dot \omega_1, \omega_1) + r_\Omega(t;\omega_1) + 2a_{\Gamma}(t;\dot \gamma_1, \gamma_1) + r_\Gamma(t;\gamma_1)\\
&\quad + 2(\beta \dot \gamma_1 - \alpha \dot \omega_1,  \beta \gamma_1 - \alpha \omega_1)_{L^2(\Gamma(t))} + \symbolForLittlec_{\Gamma}(t;\beta \gamma_1 - \alpha \omega_1, \beta \gamma_1 - \alpha \omega_1 )\\
&= 2a(t;(\dot \omega_1, \dot \gamma_1), (\omega_1,\gamma_1)) + r(t;(\omega_1,\gamma_1))\\
&= 2a(t;\dot{\mathbf{v_1}}, \mathbf{v_1}) + r(t;\mathbf{v_1}).
\end{align*}
Here, we defined 
\begin{align*}
r(t;(\omega_1,\gamma_1)) &= r_\Omega(t;\omega_1) + r_\Gamma(t;\gamma_1) + \symbolForLittlec_{\Gamma}(t;\beta \gamma_1 - \alpha \omega_1, \beta \gamma_1 - \alpha \omega_1)
\end{align*}
where $r_\Omega$ and $r_\Gamma$ are the form $r$ from \S \ref{sec:she} with domain $\Omega$ and $\Gamma$ respectively.
By the bounds on $r_\Omega,$ $r_\Gamma$ and $\lambda$, we have
\begin{align*}
|r(t;\mathbf{v_1})| 
 &\leq C_1(\norm{\omega_1}{H^1(\Omega(t))}^2 + \norm{\gamma_1}{H^1(\Gamma(t))}^2 + \norm{\beta \gamma_1 - \alpha \omega_1}{L^2(\Gamma(t))}^2)\\
&\leq C_2(\norm{\omega_1}{H^1(\Omega(t))}^2 + \norm{\gamma_1}{H^1(\Gamma(t))}^2 + \norm{\gamma_1}{L^2(\Gamma(t))}^2 + \norm{\omega_1}{L^2(\Gamma(t))}^2)\\
&\leq C_2((1+C_T^2)\norm{\omega_1}{H^1(\Omega(t))}^2 + 2\norm{\gamma_1}{H^1(\Gamma(t))}^2)\\
&\leq C_3\norm{\mathbf{v_1}}{\Vt}^2,
\end{align*}
i.e. $r(t;\cdot)$ is bounded in $\Vt$. With all the assumptions satisfied, we find from Theorem \ref{thm:existenceWithL} that there is a unique solution $(u,v) \in W(V,H)$ to the problem \eqref{eq:coupledbulksurface}.

\subsection{The dynamic boundary problem for an elliptic equation \eqref{eq:PDEstart}}\label{sec:db}
We are going to formulate the problem \eqref{eq:PDEstart} as a parabolic equation on $\Gamma(t)$. Note that $v(t)$ has a normal derivative (we expect $v(t) \in H^1(\Omega(t))$ and since $\Delta v(t) = 0$) and so we can define using \eqref{eq:formulaForNormalDerivative} the \textbf{Dirichlet-to-Neumann map} $\mathbb{A}(t)\colon H^{\frac 1 2}(\Gamma(t)) \to H^{-\frac{1}{2}}(\Gamma(t))$ (which is also bounded) by
\[\mathbb {A}(t)u(t) = \frac{\partial v(t)}{\partial \nu(t)}.\]
This map is also commonly known as the Poincar\'e--Steklov operator in the theory of boundary integral equations \cite[\S 3.7]{ss}. Now, define $\mathbb D(t)\colon H^{\frac 1 2}(\Gamma(t)) \to H^1(\Omega(t))$ by $\mathbb D(t)\tilde u = \tilde v$ where $\tilde v$ is the unique weak solution of 
\begin{equation}\label{eq:dirichletPDE}
\begin{aligned}
\lap \tilde v &= 0 &&\text{on $\Omega(t)$}\\
\tilde v &= \tilde u&&\text{on $\Gamma(t)$}
\end{aligned}
\end{equation}
given $\tilde u \in H^{\frac 12}(\Gamma(t))$. These maps give us a clue as to the spaces where we should look for solutions. Formally, we may think of a solution of the PDE \eqref{eq:PDEstart} as a pair $(v,u) \in L^2_{H^1} \times W(H^{\frac 12},H^{-\frac 12})$ such that
given $f \in L^2_{H^{-\frac 12}}$,
\begin{equation}\label{eq:PDEforuFormal}
\begin{aligned}
v &= \mathbb{D}u &&\text{in $L^2_{H^1}$}\\
\dot u + \mathbb{A}u + u &= f &&\text{in $L^2_{H^{-\frac 12}}$}\\
u(0) &= v_0&& \text{in $L^2(\Gamma_0)$}
\end{aligned}
\end{equation}
holds. Note that $(\mathbb{D}u)(t)=\mathbb{D}(t)u(t)$ for a.e. $t$. Of course, we have not defined these spaces yet so this is just formal as mentioned.
\subsubsection{Function spaces}
We use the notation and the established results of \S \ref{sec:evolvingCompactHypersurfaces}. We assume some stronger regularity on the map $\Phi^0_t$ here, namely
\begin{align*}
\Phi^0_t\colon \Gamma_0 \to \Gamma(t)\text{ is a $C^3$-diffeomorphism}\qquad\text{and} \qquad \Phi^0_{(\cdot)} \in C^3([0,T]\times \Gamma_0).
\end{align*}
In this case, we use the pivot space $H(t)=L^2(\Gamma(t))$ but now require $V(t)=H^{\frac 12}(\Gamma(t)).$
Below, we shall mainly make use of $\phi_{\Gamma, t}$ and to save space we shall write it simply as $\phi_{t}$.  We only revert to the full notation when ambiguity forces us to.

We already know that $\phi_{-t}\colon L^2(\Gamma(t)) \to L^2(\Gamma_0)$ is a well-defined linear homeomorphism. Now we show that the map $\phi_{-t}\colon H^{\frac 12}(\Gamma(t)) \to H^{\frac 12}(\Gamma_0)$ is also a linear homeomorphism. Letting $u \in H^{\frac 12}(\Gamma(t))$, it suffices to estimate only the seminorm $|\phi_{-t}u|_{H^{\frac{1}{2}}(\Gamma_0)}$:
\begin{align}
 \int_{\Gamma_0}\int_{\Gamma_0}\frac{|\phi_{-t}u(x) - \phi_{-t}u(y)|^2}{|x-y|^{n}} =  \int_{\Gamma(t)}\int_{\Gamma(t)}\frac{|u(x_t) - u(y_t)|^2}{|\Phi^t_0(x_t)-\Phi^t_0(y_t)|^{n}}J_0^t(x_t) J_0^t(y_t)\label{eq:prelim1}
\end{align}
where we made the substitutions $x_t = \Phi^0_t(x) \in \Gamma(t)$ and $y_t = \Phi^0_t(y) \in \Gamma(t)$. Since $\Phi_t^0$ is a  $C^1$-diffeomorphism between compact spaces, it is bi-Lipschitz with Lipschitz constant $C_L$ independent of $t$ (because the spatial derivatives of $\Phi_t^0$ are uniformly bounded). This implies $|x_t - y_t| \leq C_L|\Phi_0^t(x_t) - \Phi_0^t(y_t)|$ so that \eqref{eq:prelim1} becomes
\begin{align*}
|\phi_{-t}u|_{H^{\frac 12}(\Gamma_0)}^2 \leq C_L^{n}C_J^2\int_{\Gamma(t)}\int_{\Gamma(t)}\frac{|u(x_t) - u(y_t)|^2}{|x_t-y_t|^{n}} = C_L^{n}C_J^2|u|_{H^{\frac 12}(\Gamma(t))}^2,
\end{align*}
where we used the uniform bound on $J_0^t$. So we have the uniform bound
\[\norm{\phi_{-t}u}{H^{\frac 12}(\Gamma_0)} \leq C\norm{u}{H^{\frac 12}(\Gamma(t))}.\]
A similar bound holds for the operator $\phi_t$ by the same arguments as above since $\Phi^t_0 = (\Phi^0_t)^{-1}$ also satisfies the same properties as above. It follows by the smoothness on $\Phi^0_{(\cdot)}$ that $J_{(\cdot)}^0 \in C^2([0,T]\times {\Gamma_0}).$ This implies that $J_t^0\colon \Gamma_0 \to \mathbb{R}$ is (globally) Lipschitz (see the paragraph after the proof of Proposition 2.4 in \cite{hebeynonlinear}).

The map 
\begin{align*}
t \mapsto |{\phi_t u}|_{H^{\frac 12}(\Gamma(t))}^2 &= \int_{\Gamma(t)}\int_{\Gamma(t)}\frac{|\phi_t u(x) - \phi_t u(y)|^2}{|x-y|^{n}} = \int_{\Gamma_0}\int_{\Gamma_0}\frac{|u(x_0) - u(y_0)|^2}{|\Phi^0_t(x_0)-\Phi^0_t(y_0)|^{n}}J_t^0(x_0)J_t^0(y_0)
\end{align*}
is continuous. To see this, define the integrand
\[g(x_0,y_0,t) = \frac{|u(x_0) - u(y_0)|^2}{|\Phi^0_t(x_0)-\Phi^0_t(y_0)|^{n}}J_t^0(x_0)J_t^0(y_0).\]
Now, $t \mapsto g(x_0,y_0,t)$ is continuous for almost all $(x_0,y_0)$ (it only fails when the denominator is zero, where $x_0 = y_0$, and the set of such points has zero measure), and we have the domination $g(x_0,y_0,t) \leq h(x_0,y_0)$ for all $t$ and almost all $(x_0,y_0)$ by an integrable function $h$; this follows due to the smoothness assumptions on $\Phi^0_{(\cdot)}$ and $J^0_{(\cdot)}$. Therefore, $t \mapsto \int_{\Gamma_0}\int_{\Gamma_0}g(x_0,y_0,t)$ is continuous. This enables us to conclude that $(H, \phi_{(\cdot)})$ and $(V, \phi_{(\cdot)}|_V)$ are compatible.

There is some effort needed in order to show the evolving space equivalence. We start with the following two results which are used continually.

\begin{lem}
For $y \in \Gamma_0$, we have
$$\int_{\Gamma_0}\frac{1}{|x-y|^{n-2}}\;\mathrm{d}\sigma(x) < C$$
where $C$ does not depend on $y.$
\end{lem}

This lemma can be proved by first setting $y=0$ (without loss of generality) and then splitting the domain of integration into two sets, one of which is a ball centred at the origin. The integral over the ball can be tackled with the assumption of the domain being Lipschitz and switching to polar coordinates, while the integral over the complement of the ball is obviously finite.
\begin{lem}\label{lem:algebraHS}If $\rho \in C^1(\Gamma_0)$ and $u \in H^{\frac 12}(\Gamma_0)$ then $\rho u \in H^{\frac 12}(\Gamma_0)$ and
\begin{equation}\label{eq:algebraHS}
\norm{\rho u}{H^{\frac 12}(\Gamma_0)} \leq  C\norm{\rho}{C^1(\Gamma_0)}\norm{u}{H^{\frac 12}(\Gamma_0)}
\end{equation}
where $C$ does not depend on $\rho$ or $u$.
\end{lem}
\begin{proof}
Note that $\rho$ and $\grad \rho$ are bounded from above and $\rho$ is Lipschitz. We begin with
\begin{align*}
\norm{\rho u}{H^{\frac 12}(\Gamma_0)}^2 
&\leq \norm{\rho}{C^0(\Gamma_0)}^2\norm{u}{L^2(\Gamma_0)}^2 + \int_{\Gamma_0}\int_{\Gamma_0}\frac{|\rho(x) u(x) - \rho(y) u(y)|^2}{|x-y|^{n}}\;\mathrm{d}x\mathrm{d}y.
\end{align*}
The last term is
\begin{align*}
\int_{\Gamma_0}&\int_{\Gamma_0}\frac{|\rho(x) u(x) - \rho(y) u(y)|^2}{|x-y|^{n}}\\ 
&\leq 2\int_{\Gamma_0}\int_{\Gamma_0}\frac{|\rho(x)|^2 |u(x) - u(y)|^2}{|x-y|^{n}} + 2\int_{\Gamma_0}\int_{\Gamma_0}\frac{|u(y)|^2|\rho(x) - \rho(y) |^2}{|x-y|^{n}}\\
&\leq 2\norm{\rho}{C^0(\Gamma_0)}^2|u|_{H^{\frac 12}(\Gamma_0)}^2 + 2\norm{\grad \rho}{C^0(\Gamma_0)}^2\int_{\Gamma_0}\int_{\Gamma_0}\frac{|u(y)|^2}{|x-y|^{n-2}}.
\end{align*}
Using the previous lemma, the integral in the second term is
\begin{align*}
\int_{\Gamma_0}\int_{\Gamma_0}\frac{|u(y)|^2}{|x-y|^{n-2}} &= \int_{\Gamma_0}|u(y)|^2\int_{\Gamma_0}{|x-y|^{2-n}} \leq C_1\norm{u}{L^2(\Gamma_0)}^2.
\end{align*}
Putting it all together, we achieve \eqref{eq:algebraHS}.
\end{proof}
In the following lemmas, let $J \in C^2([0,T]\times \Gamma_0)$.
\begin{lem}
If $\psi \in \mathcal{D}((0,T);H^{\frac 12}(\Gamma_0))$, then $\psi J \in \mathcal W(V_0,V_0^*)$ and $(\psi J)' = \psi' J + \psi J'$.
\end{lem}
\begin{proof}
Let us note that
\begin{align*}
\psi \in C^0([0,T];H^{\frac 12}(\Gamma_0))\quad\text{and}\quad J \in C^0([0,T];H^{\frac 12}(\Gamma_0))\cap C^1([0,T];C^1(\Gamma_0)).
\end{align*}
{The first part of the second inclusion holds because $J \in C^0([0,T];H^1(\Gamma_0))$ and because $H^1(\Gamma_0) \subset H^{\frac 12}(\Gamma_0)$ is continuous \cite[Theorem 2.5.1 and Theorem 2.5.5]{ss}. The uniform continuity of $J$ over the compact set $[0,T]\times \Gamma_0$ gives the second part.}

Now, observe that $\psi(t)J(t) \in H^{\frac 12}(\Gamma_0)$ for all $t$ by Lemma \ref{lem:algebraHS}. To see that $\psi J \in C^0([0,T];H^{\frac 12}(\Gamma_0))$, fix an arbitrary $t \in [0,T]$, let $t_n \to t$ and consider
\begin{align*}
\norm{\psi(t)J(t)-\psi(t_n)J(t_n)}{H^{\frac 12}(\Gamma_0)} &\leq \norm{\psi(t)(J(t)-J(t_n))}{H^{\frac 12}(\Gamma_0)}+\norm{J(t_n)(\psi(t)-\psi(t_n))}{H^{\frac 12}(\Gamma_0)}\\
&\leq C{\norm{J(t)-J(t_n)}{C^1(\Gamma_0)}}\norm{\psi(t)}{H^{\frac 12}(\Gamma_0)}\\
&\quad+C\norm{J(t_n)}{C^1(\Gamma_0)}\norm{\psi(t)-\psi(t_n)}{H^{\frac 12}(\Gamma_0)}.
\end{align*}
The first of these terms tends to zero as $t_n \to t$ because $J \in C^0([0,T]; C^1(\Gamma_0))$ and the second because $\psi \in C^0([0,T];H^{\frac 12}(\Gamma_0))$ in addition to the aforementioned smoothness of $J$.

Now we show that in fact $\psi J$ is (classically) differentiable and that $(\psi J)'=\psi'J + \psi J'$. Observe that $\psi'(t)J(t) + \psi(t)J'(t) \in H^{\frac 12}(\Gamma_0)$ by Lemma \ref{lem:algebraHS}.
Define the difference quotient $D^hJ(t) = (J(t+h)-J(t))/h$ and $D^h\psi(t)$ similarly and note that
\begin{align*}
\nonumber &\norm{\frac{\psi(t+h)J(t+h)-\psi(t)J(t)}{h} - \psi'(t)J(t) - \psi(t)J'(t)}{H^{\frac 12}(\Gamma_0)}\\
&\leq \norm{\psi(t+h)D^hJ(t)- \psi(t)J'(t)}{H^{\frac 12}(\Gamma_0)} +\norm{D^h\psi(t)J(t) - \psi'(t)J(t)}{H^{\frac 12}(\Gamma_0)}\label{eq:longEquation}\\
&\leq C\norm{D^hJ(t) - J'(t)}{C^1(\Gamma_0)}\norm{\psi(t+h)}{H^{\frac 12}(\Gamma_0)}+ C\norm{J'(t)}{C^1(\Gamma_0)}\norm{\psi(t+h)-\psi(t)}{H^{\frac 12}(\Gamma_0)}\\
&\quad + C\norm{J(t)}{C^1(\Gamma_0)}\norm{D^h\psi(t) - \psi'(t)}{H^{\frac 12}(\Gamma_0)}.
\end{align*}
In the above, we used
\begin{align*}
\norm{\psi(t+h)D^hJ(t)- \psi(t)J'(t)}{H^{\frac 12}(\Gamma_0)} &\leq \norm{\psi(t+h)\left(D^hJ(t) - J'(t)\right)}{H^{\frac 12}(\Gamma_0)}\\
&\quad+ \norm{(\psi(t+h)-\psi(t))J'(t)}{H^{\frac 12}(\Gamma_0)}.
\end{align*}
It follows that $\norm{D^hJ(t) - J'(t)}{C^1(\Gamma_0)} \to 0$ because $J \in C^1([0,T];C^1(\Gamma_0))$. Thus, we find
\[\lim_{h \to 0}\norm{\frac{\psi(t+h)J(t+h)-\psi(t)J(t)}{h} - \psi'(t)J(t) - \psi(t)J'(t)}{H^{\frac 12}(\Gamma_0)} = 0.\]
{This proves the product rule for $(\psi J)'$. We finish by proving that $(\psi J)' \in C^0([0,T];H^{\frac 12}(\Gamma_0))$. Fix again $t \in [0,T]$ and let $t_n \to t$. Observe that}
\begin{align*}
&\norm{\psi'(t_n)J(t_n) + \psi(t_n)J'(t_n) - \psi'(t)J(t) - \psi(t)J'(t)}{H^{\frac 12}(\Gamma_0)}\\
&\leq \norm{\psi'(t_n)(J(t_n)-J(t))}{H^{\frac 12}(\Gamma_0)} + \norm{J(t)(\psi'(t_n)-\psi'(t))}{H^{\frac 12}(\Gamma_0)}\\ &\quad+ \norm{\psi(t_n)(J'(t_n)-J'(t))}{H^{\frac 12}(\Gamma_0)}+ \norm{J'(t)(\psi(t_n)-\psi(t))}{H^{\frac 12}(\Gamma_0)}\\
&\leq C\norm{\psi'(t_n)}{H^{\frac 12}(\Gamma_0)}\norm{J(t_n)-J(t)}{C^1(\Gamma_0)} + C\norm{J(t)}{C^1(\Gamma_0)}\norm{\psi'(t_n)-\psi'(t)}{H^{\frac 12}(\Gamma_0)}\\ &\quad + C\norm{\psi(t_n)}{H^{\frac 12}(\Gamma_0)}\norm{J'(t_n)-J'(t)}{C^1(\Gamma_0)} + C\norm{J'(t)}{C^1(\Gamma_0)}\norm{\psi(t_n)-\psi(t)}{H^{\frac 12}(\Gamma_0)}
\end{align*}
{and this tends to zero because $J \in C^1([0,T];C^1(\Gamma_0))$ and $\psi \in C^1([0,T];H^{\frac 12}(\Gamma_0))$. All in all, we have shown that $\psi J \in C^1([0,T];H^{\frac 12}(\Gamma_0)) \subset \mathcal{W}(V_0, V_0^*)$.}
\end{proof}
\begin{lem}
For every $u \in \mathcal{W}(V_0,V_0^*)$, $J u \in \mathcal W(V_0,V_0^*)$.
\end{lem}
\begin{proof}
Let $\psi \in \mathcal{D}((0,T);H^{\frac 12}(\Gamma_0))$ and for $u \in \mathcal W(V_0,V_0^*)$, consider
\begin{align*}
\int_0^T \langle u'(t), J(t) \psi(t) \rangle_{H^{-\frac 12}(\Gamma_0), H^{\frac 12}(\Gamma_0)}
&= -\int_0^T (J'(t) \psi(t) + J(t) \psi'(t), u(t))_{L^2(\Gamma_0)}\tag{by integration by parts and the last lemma}\\
&=-\int_0^T (\psi(t), J'(t)u(t))_{L^2(\Gamma_0)}- \int_0^T (\psi'(t), J(t) u(t))_{L^2(\Gamma_0)}.
\end{align*}
A rearrangement yields
\[\int_0^T (J(t) u(t), \psi'(t))_{L^2(\Gamma_0)} = -\int_0^T \langle J'(t) u(t) + J(t) u'(t), \psi(t) \rangle_{H^{-\frac 12}(\Gamma_0),H^{\frac 12}(\Gamma_0)}.\]
This shows that $J u$ has a weak derivative, and $(Ju)' \in L^2(0,T;H^{-\frac 12}(\Gamma_0))$ since we have $J'u \in L^2(0,T;H^{\frac 12}(\Gamma_0))$ and $J u' \in L^2(0,T;H^{-\frac 1 2}(\Gamma_0))$.
\end{proof}
\begin{thm}The evolving space equivalence between $\mathcal{W}(V_0,V_0^*)$ and $W(V,V^*)$ holds.
\end{thm}
\begin{proof}
The last result shows that if $u \in \mathcal{W}(V_0,V_0^*)$ then $J_t^0u \in \mathcal{W}(V_0,V_0^*)$. Because $1/J_t^0 \in C^2([0,T]\times {\Gamma_0}),$ the converse also holds. Since
\begin{align*}
(J_t^0u(t))' &=  J_t^0u'(t) + \hat{\symbolForBigC}(t)u(t),
\end{align*}
we have (in the notation of Theorem \ref{thm:spaceW1}) $\hat{S}(t) = T_t=J_t^0$ and $\hat{D}(t) \equiv 0,$ and it follows that $\hat{S}(\cdot)u'(\cdot) \in L^2(0,T;H^{-{\frac 1 2}}(\Gamma_0)).$ Thus Theorem \ref{thm:spaceW1} can be applied.
\end{proof}

\subsubsection{Weak formulation and well-posedness}
Now that we have defined some notation and function spaces, the equation \eqref{eq:PDEforuFormal} has a precise meaning and we can define a notion of solution.
\begin{defn}
{With $H^1 = \{H^1(\Omega(t))\}_{t \in [0,T]}$, given $f \in L^2_{V^*}$, a solution of \eqref{eq:PDEstart} is a pair $(v,u) \in L^2_{H^1} \times W(V,V^*)$ such that}
\begin{equation}\label{eq:PDEforu}
\begin{aligned}
v &= \mathbb{D}u &&\text{in $L^2_{H^1}$}\\
\dot u + \mathbb{A}u + u &= f &&\text{in $L^2_{V^*}$}\\
u(0) &= v_0&& \text{in $H_0$}.
\end{aligned}
\end{equation}
\end{defn}
Note that the first condition implies $\lap_t v(t) = 0$ and $v(t)|_{\Gamma(t)} = u(t)$ for almost every $t$. We need the following auxiliary result.
\begin{lem}\label{lem:dirichletPDE}
The map $\mathbb D(t)\colon H^{\frac 1 2}(\Gamma(t)) \to H^1(\Omega(t))$ is uniformly bounded:
\begin{equation}\label{eq:boundOnD}
\norm{\mathbb D(t)\tilde u}{H^1(\Omega(t))} \leq C\norm{\tilde u}{H^{\frac 1 2}(\Gamma(t))} \quad \forall \tilde u \in H^{\frac 12}(\Gamma(t))
\end{equation}
where the constant $C$ does not depend on $t \in [0,T]$.
\end{lem}

To prove this lemma, we need the following results which show that certain standard results are in a sense uniform in $t \in [0,T]$. The method of proof of the next lemma is identical to that of Lemma \ref{lem:traceCommutesWithPhiL2}.

\begin{lem}\label{lem:traceCommutesWithPhi}
Let $\tau_t\colon H^1(\Omega(t)) \to H^{\frac 12}(\Gamma(t))$ denote the trace map. For all $v \in H^1(\Omega_0),$ the equality $\tau_t(\phi_{\Omega,t} v) = \phi_{\Gamma,t}(\tau_0v)$ holds in $H^{\frac 1 2}(\Gamma(t))$.
\end{lem}

\begin{lem}For each $t \in [0,T]$, we have
\begin{align}
\norm{v}{H^1(\Omega(t))} &\leq C_1\norm{\grad v}{L^2(\Omega(t))}&&\text{$\forall v \in H^1_0(\Omega(t))$}\label{eq:poincareInequality}\\
\norm{\grad v}{L^2(\Omega(t))}^2 +  \norm{v}{L^2(\Gamma(t))}^2
&\geq C_2\norm{v}{H^1(\Omega(t))}^2 &&\text{$\forall v \in H^1(\Omega(t))$}\label{eq:compactnessInequality}\\
\inf_{\substack{v \in H^1(\Omega(t))\\ \tau_t v = u}}\norm{v}{H^1(\Omega(t))} &\leq C_3\norm{u}{H^{\frac 1 2}(\Gamma(t))}&&\text{$\forall u \in H^{\frac 12}(\Gamma(t))$}\label{eq:equivalenceOfNormsInf}\\
\norm{\tau_t v}{H^{\frac 12}(\Gamma(t))} &\leq C_4\norm{v}{H^{1}(\Omega(t))}&&\text{$\forall v \in H^{1}(\Omega(t))$}\label{eq:traceInequalityH12}
\end{align}
where $C_1$, $C_2,$ $C_3,$ and $C_4$ do not depend on $t$.
\end{lem}
The strategy to prove this lemma is to start with each respective inequality at $t=0$, in which case: \eqref{eq:poincareInequality} is the Poincar\'e inequality on $\Omega_0$, \eqref{eq:compactnessInequality} follows by a compactness argument, \eqref{eq:equivalenceOfNormsInf} is an equivalence of norms and \eqref{eq:traceInequalityH12} is the trace inequality on $\Omega_0$. Then for \eqref{eq:poincareInequality}, use the chain rule $\grad(\phi_{-t}v) = \grad(v \circ \Phi^0_t) = \phi_{-t}(\grad v)\mathbf{D}\Phi^0_t$ and the uniform boundedness of $\mathbf{D}\Phi^0_t$. The inequality \eqref{eq:compactnessInequality} is obtained with the identity $\grad v = \grad (\phi_{-t}\phi_tv) = \phi_{-t}(\grad \phi_t v)\mathbf{D}\Phi^0_t$ and Lemma \ref{lem:traceCommutesWithPhi}. The lemma is also the key ingredient to show \eqref{eq:equivalenceOfNormsInf} and \eqref{eq:traceInequalityH12} (see the discussion in \S \ref{sec:evolvingFlatHypersurfaces} for how to prove the latter).
\begin{proof}[Proof of Lemma \ref{lem:dirichletPDE}]
We prove the well-posedness of \eqref{eq:dirichletPDE} in addition to the uniform bound \eqref{eq:boundOnD} for the convenience of the reader. First, we use the trace map $\tau_t\colon H^{1}(\Omega(t)) \to H^{\frac 1 2}(\Gamma(t))$ to see that there is a function $\tilde{v}_{\tilde{u} }\in H^1(\Omega(t))$ such that $\tau_t\tilde{v}_{\tilde{u}} = \tilde{u}$. With $\tilde v = \mathbb{D}(t)\tilde u$, set $d := \tilde v - \tilde{v}_{\tilde{u}}$. Then 
$d$ solves
\begin{equation}\label{eq:dirichletPDEhomogeneous}
\begin{aligned}
\lap d &= -\lap \tilde v_{\tilde u} &&\text{on $\Omega(t)$}\\
d &= 0&&\text{on $\Gamma(t)$}.
\end{aligned}
\end{equation}
Define $b_t(\cdot,\cdot)\colon H^1(\Omega(t)) \times H^1(\Omega(t)) \to \mathbb{R}$ and $l_t(\cdot)\colon H^1(\Omega(t)) \to \mathbb{R}$ by
\begin{align*}
b_t(d, \varphi) = \int_{\Omega(t)}\grad d \grad \varphi\qquad\text{and}\qquad l_t(\varphi) = \int_{\Omega(t)} \grad \tilde w_{\tilde u} \grad \varphi.
\end{align*}
Clearly $l_t$ and $b_t$ are bounded and the Poincar\'e inequality \eqref{eq:poincareInequality} implies that $b_t$ is coercive with the coercivity constant $C_P^{-1}$ independent of $t$. 
By Lax--Milgram, there is a unique solution $d \in H^1_0(\Omega(t))$ to \eqref{eq:dirichletPDEhomogeneous} 
satisfying
\begin{align*}
\norm{d}{H^1(\Omega(t))} &\leq C_P\norm{\tilde v_{\tilde u}}{H^1(\Omega(t))}.
\end{align*}
Because this inequality holds for all lifts $\tilde v_{\tilde u}$ of $\tilde u$ we must have
\begin{align*}
\norm{d}{H^1(\Omega(t))} &\leq C_P\inf_{\substack{w \in H^1(\Omega(t)),\\ \tau_t w = \tilde u}}\norm{w}{H^1(\Omega(t))}\\
&\leq C_1\norm{\tilde u}{H^{\frac 1 2}(\Gamma(t))}
\end{align*}
where the second inequality is thanks to \eqref{eq:equivalenceOfNormsInf}. Since $\tilde v = d + \tilde v_{\tilde u}$, we see that \eqref{eq:dirichletPDE} has a unique solution $\tilde v \in H^1(\Omega(t))$ with
\[\norm{\tilde v}{H^1(\Omega(t))} \leq C_2\norm{\tilde u}{H^{\frac 1 2}(\Gamma(t))}\]
due to the arbitrariness of the lift $\tilde v_{\tilde u}$.
\end{proof}
Now we can conclude the well-posedness of \eqref{eq:PDEforu} by checking the assumptions on $\mathbb{A}$. With $w \in L^2_V$ and using \eqref{eq:formulaForNormalDerivative},
\[\langle \mathbb  A(t)u(t), w(t) \rangle_{H^{-\frac 12}(\Gamma(t)), H^{\frac 12}(\Gamma(t))}= \int_{\Omega(t)}\grad (\mathbb D(t)u(t))\grad (\mathbb{E}(t)w(t)).\]
So the bilinear form $a(t;\cdot,\cdot)\colon H^{\frac 12}(\Gamma(t)) \times H^{\frac 12}(\Gamma(t)) \to \mathbb{R}$ is
\[a(t;u,w) := \int_{\Omega(t)}\grad (\mathbb D(t)u)\grad (\mathbb{E}(t)w) +  \int_{\Gamma(t)}uw.\]
We take $\mathbb E = \mathbb D$, and we obtain by the uniform bound \eqref{eq:boundOnD} the boundedness of $a(t;\cdot,\cdot)$:
\begin{align*}
|a(t;u,w)| &\leq \norm{\mathbb D(t)u}{H^1(\Omega(t))}\norm{\mathbb D(t)w}{H^1(\Omega(t))} + \norm{u}{L^2(\Gamma(t))}\norm{w}{L^2(\Gamma(t))}\\
&\leq C_D^2\norm{u}{H^{\frac 1 2}(\Gamma(t))}\norm{w}{H^{\frac 1 2}(\Gamma(t))} + \norm{u}{L^2(\Gamma(t))}\norm{w}{L^2(\Gamma(t))}\\
&\leq (C_D^2+1)\norm{u}{H^{\frac 1 2}(\Gamma(t))}\norm{w}{H^{\frac 1 2}(\Gamma(t))}.
\end{align*}
For coercivity, 
\begin{align*}
a(t;w,w) 
&=  \norm{\grad (\mathbb D(t)w)}{L^2(\Omega(t))}^2 +  \norm{w}{L^2(\Gamma(t))}^2\tag{again with $\mathbb E=\mathbb D$}\\
&\geq C_1\norm{\mathbb D(t)w}{H^1(\Omega(t))}^2\tag{using \eqref{eq:compactnessInequality}}\\
&\geq C_2\norm{w}{H^{\frac 1 2}(\Gamma(t))}^2
\end{align*}
by the trace inequality \eqref{eq:traceInequalityH12}. Therefore, we have a unique solution $u \in W(H^{\frac{1}{2}}, H^{-\frac{1}{2}})$ to \eqref{eq:PDEforu}, and with $v(t) := \mathbb{D}(t)u(t)$ and the uniform bound \eqref{eq:boundOnD}, we find $(v,u)$ to be a solution of \eqref{eq:PDEstart}.